\documentclass[a4paper, 11pt,reqno]{amsart}

\usepackage[T1]{fontenc}      % Uses 'T1' font encoding.
\usepackage[foot]{amsaddr}    % Manages options for displayed author information.

\usepackage[british]{babel}   % Support for different languages.
\usepackage{csquotes}         % Adds quotation marks according to language.
\usepackage[babel]{microtype} % Makes things look nicer (spacing and so on).

\usepackage{mathtools}        % Extension of 'amsmath' which fixes some bugs.
\usepackage{amssymb}          % Extra symbols and fonts.
\usepackage{amsthm}           % Helps to define "theorem-like" environments.
\usepackage{amstext}          % Defines a '\text' macro for math environments.
\usepackage{mleftright}       % Fixes spacing issue.
\usepackage{gensymb}          % Allows to use degrees.
\usepackage[makeroom]{cancel} % Defines cancellation of terms in formulae.
\usepackage{mathdots}         % Adds some definitions of dots.
\usepackage{mathrsfs}         % Adds rsfs fonts.
\usepackage{stickstootext}
\usepackage[stickstoo,varbb]{newtxmath} % These two lines are an alternative for stix2 which do essentially the same but fix several issues with some symbols...
\makeatletter
\DeclareFontFamily{U}{ntxmia}{\skewchar \font =127}
 \DeclareFontShape{U}{ntxmia}{m}{it}{
                        <-> \ntxmath@scaled ntxmia
                      }{}    
                      \DeclareFontShape{U}{ntxmia}{b}{it}{
                        <-> \ntxmath@scaled ntxbmia
                      }{}
\makeatother

\usepackage{graphicx}         % Provides commands to include graphics.
\usepackage{psfrag}           % Gives some extra useful things for graphics.
\usepackage{caption}          % Options for formatting float captions.
\usepackage{tikz}             % Language to create graphics programmatically.
\usetikzlibrary{backgrounds}

\usepackage{booktabs}         % Enhances the quality of tables.

\usepackage{enumitem}         % Allows customization for enumerations.

\usepackage[margin=1in]{geometry} % Provides commands to define the page layout.

\usepackage[textsize=footnotesize]{todonotes}
\usepackage[numbers,sort]{natbib}

\makeatletter
\def\NAT@spacechar{~}% NEW
\makeatother

\usepackage{hyperref}              % For hyperlinks within the document
\usepackage[capitalise]{cleveref}  % For improved referencing options; use [capitalise] if desired
\hypersetup{colorlinks=true,
   citecolor=blue,
   filecolor=blue,
   linkcolor=blue,
   urlcolor=blue
}
\usepackage{url}

\makeatletter
\if@cref@capitalise
\crefname{figure}{Figure}{Figures}
\crefname{claim}{Claim}{Claims}
\else
\crefname{figure}{figure}{figures}
\crefname{claim}{claim}{claims}
\fi
\makeatother
\Crefname{figure}{Figure}{Figures}
\Crefname{claim}{Claim}{Claims}

\usepackage{algorithm}
\usepackage{algpseudocode}

\newtheorem{definition}{Definition}[section]
\newtheorem{claim}{Claim}

\newtheorem{proposition}[definition]{Proposition}
\newtheorem{theorem}[definition]{Theorem}
\newtheorem{corollary}[definition]{Corollary}
\newtheorem{lemma}[definition]{Lemma}

\theoremstyle{definition}
\newtheorem{remark}[definition]{Remark}

\newenvironment{claimproof}{%
\let\origqed=\qedsymbol%
\renewcommand{\qedsymbol}{$\blacktriangleleft$}%
\begin{proof}}{\end{proof}\let\qedsymbol=\origqed}

\numberwithin{equation}{section}

\renewcommand{\binom}[2]{\ensuremath{\mleft(\kern-.1em\genfrac{}{}{0pt}{}{#1}{#2}\kern-.1em\mright)}}    % This makes binomial numbers nicer with stix2 (in displayed equations). Remove if stix2 is not loaded.
\newcommand{\inbinom}[2]{\ensuremath{\bigl(\kern-.1em\genfrac{}{}{0pt}{}{#1}{#2}\kern-.1em\bigr)}} % This is better for inline equations, as it will keep sizes of parentheses consistent and not create extra vertical space.

\newcommand*\nume{\ensuremath{\mathrm{e}}}

\newcommand{\cS}{\mathcal{S}}

\newcommand{\NN}{\mathbb{N}}
\newcommand{\PP}{\mathbb{P}}

\newcommand{\del}{\partial}

\newcommand{\mix}{\mathrm{mix}}
\makeatletter
\def\moverlay{\mathpalette\mov@rlay}
\def\mov@rlay#1#2{\leavevmode\vtop{%
  \baselineskip\z@skip \lineskiplimit-\maxdimen
  \ialign{\hfil$\m@th#1##$\hfil\cr#2\crcr}}}
\newcommand{\charfusion}[3][\mathord]{
    #1{\ifx#1\mathop\vphantom{#2}\fi
        \mathpalette\mov@rlay{#2\cr#3}
      }
    \ifx#1\mathop\expandafter\displaylimits\fi}
\makeatother

\renewcommand{\deg}{\operatorname{deg}}

\newcommand{\eps}{\epsilon}

\newcommand{\COMMENT}[1]{}
\renewcommand{\COMMENT}[1]{\footnote{\textcolor{blue!70!black}{#1}}} % comment out to hide comments
\newcommand{\COMNEW}[1]{}
\renewcommand{\COMNEW}[1]{\footnote{\textcolor{red!70!black}{#1}}} % comment out to hide comments

\title{Speeding up random walk mixing by starting from a uniform vertex}

\author[A.~Espuny D\'iaz]{Alberto Espuny D\'iaz}
\email{alberto.espuny-diaz@tu-ilmenau.de}
\address[Espuny D\'iaz]{Institut f\"ur Mathematik, Technische Universit\"at Ilmenau, 98684 Ilmenau, Germany.}
\author[P.~Morris]{Patrick Morris}
\email{pmorrismaths@gmail.com}
\address[Morris, Peraranau, Serra]{Departament de Matem\`atiques and IMTECH, Universitat Polit\`ecnica de Catalunya (UPC), Barcelona, Spain.}
\author[G.~Perarnau]{Guillem Perarnau}
\email{guillem.perarnau@upc.edu}
\address[Perarnau, Serra]{Centre de Recerca Matem\`atica, Barcelona, Spain.}
\author[O.~Serra]{Oriol Serra}
\email{oriol.serra@upc.edu}

\thanks{This research has been supported by the Spanish Agencia Estatal de Investigación under projects PID2020-113082GB-I00 and the Severo Ochoa and María de Maeztu Program for Centers and Units of Excellence in R\&{}D (CEX2020-001084-M). 
Alberto Espuny Díaz was partially supported by the Carl Zeiss Foundation and by DFG (German
Research Foundation) grant PE~2299/3-1.
Patrick Morris was supported by the DFG Walter Benjamin program - project number 504502205.
}

\date{\today}

\begin{document}

\begin{abstract}
The theory of rapid mixing random walks plays a fundamental role in the study of modern randomised algorithms. 
Usually, the mixing time is measured with respect to the worst initial position.
It is well known that the presence of bottlenecks in a graph hampers mixing and, in particular, starting inside a small bottleneck significantly slows down the diffusion of the walk in the first steps of the process.
The average mixing time is defined to be the mixing time starting at a uniformly random vertex and hence is not sensitive to the slow diffusion caused by these bottlenecks.

In this paper we provide a general framework to show logarithmic average mixing time for random walks on graphs with small bottlenecks.
The framework is especially effective on certain families of random graphs with heterogeneous properties.
We demonstrate its applicability on two random models for which the mixing time was known to be of order $(\log n)^2$, speeding up the mixing to order $\log n$.
First, in the context of smoothed analysis on connected graphs, we show logarithmic average mixing time for randomly perturbed graphs of bounded degeneracy.
A particular instance is the Newman-Watts small-world model.
Second, we show logarithmic average mixing time for supercritically percolated expander graphs.
When the host graph is complete, this application gives an alternative proof that the average mixing time of the giant component in the supercritical Erd\H os-R\'enyi graph is logarithmic.
\end{abstract}
\maketitle

\section{Introduction}\label{sec:intro}

Random walks on graphs are one of the fundamental tools for sampling (see, e.g.,~\cite{R03}).
Applications are numerous in areas such as computer science, discrete mathematics and statistical physics.
Prominent examples include the polynomial-time algorithm to estimate the volume of a convex body~\cite{DFK91}, computing the matrix permanent~\cite{JS89} or the use of Glauber dynamics to sample from Gibbs distributions, in particular from proper colourings~\cite{V00}. 

Most usually, the size of the sampling space is exponential in the input size, and fully exploring this space is computationally intractable. 
The Markov chain Monte Carlo (MCMC) method consists of running a random walk in an appropriately chosen graph, whose vertex set is the sample space, until its distribution is arbitrarily close to equilibrium, regardless of the initial state.
At that time we say the walk has mixed, and the time until it does is called the (worst-case) \emph{mixing time}.
To obtain efficient sampling algorithms it suffices to prove that the mixing time is poly-logarithmic in the input size.

The connection between rapid mixing and expanders is well-established.
In the context of random walks, expansion is measured by means of a graph parameter called \emph{conductance}; see \cref{sec:conductance} for the precise definition.
Jerrum and Sinclair~\cite{JS89} gave an upper bound on the mixing time depending on the conductance and the logarithm of the minimum stationary value.
This bound is central in the theory of Markov chains.

Random environments are particularly interesting sampling spaces and, in the last 20 years, researchers have developed the theory of random walks on random graphs.
As expected, the good expansion properties of random graphs ensure rapid mixing.
By the Jerrum-Sinclair bound, graphs with conductance bounded away from zero mix in logarithmically many steps and usually exhibit \emph{cut-off}, that is, the distribution converges rapidly to the stationary distribution in a small window of time.
Good examples are random graph models with control on the degrees, such as random regular graphs~\cite{LS10}, random graphs with given degree sequences~\cite{BLPS18,BS17}, their directed analogues~\cite{BCS18,CCPQ22}, or graphs perturbed by random perfect matchings~\cite{HSS22}.

Nonetheless, the presence of small obstructions  slows down the mixing.
A canonical example is the giant component of a sparse Erd\H os-R\'enyi graph $G(n,c/n)$ with $c>1$.
This component contains relatively small \emph{bottlenecks}, that is, connected sets that only have few edges connecting them to the rest of the graph. 
In such cases, tools like the Jerrum-Sinclair bound fail to pin down the correct order of the mixing time.
Fountoulakis and Reed~\cite{FR07} introduced a strengthening of the bound that is sensitive to small bottlenecks and used it to show that the mixing time of the largest component in $G(n,c/n)$ is asymptotically almost surely (a.a.s.\ for short) $O(\log^2 n)$~\cite{FR08}.
Indeed, this is the correct order as the component contains paths of degree $2$ vertices (also referred to as \emph{bare paths}) whose length is of order $\log{n}$.
Starting at the centre of such paths, a random walk takes $\Omega(\log^2 n)$ steps in expectation to escape from it.
We remark that the mixing time in the supercritical random graph $G(n,c/n)$ was also bounded independently by Benjamini, Kozma and Wormald~\cite{BKW14}, using a different approach investigating the \emph{anatomy} of the giant component. 

However, these local bottlenecks are a negligible part of the giant component and the rest of the component has good expansion properties.
This suggests that, if the random walk started outside the bottlenecks, the mixing time would decrease.
This was implicit in the work of Benjamini, Kozma and Wormald~\cite{BKW14} and their description of the giant component, and such a speeding up of mixing time was also conjectured explicitly by Fountoulakis and Reed~\cite{FR08}. 
Berestycki, Lubetzky, Peres and Sly~\cite{BLPS18} confirmed their prediction, showing that there exists $a=a(c)$ such that the mixing time starting at a uniformly random vertex  is asymptotically $a \log{n}$ with high probability (they in fact proved much more, establishing the value of $a(c)$ precisely as well as cut-off for the random walk).
This result reinforces the idea that, in certain heterogeneous scenarios, averaging over the starting position yields more efficient sampling algorithms.

The goal of this paper is to provide a general framework to show logarithmic average-case mixing time for random walks on graphs with small bottlenecks. 

\subsection{Average mixing times} \label{sec:intro-main}

Given an $n$-vertex graph $G$, the \emph{lazy random walk} over $G$ is a Markov chain with state space $V(G)$ which can be defined as follows.
If at any given time we are in a vertex $u\in V(G)$, the lazy random walk stays in $u$ with probability $1/2$, and with probability $1/2$ it moves to a uniformly random neighbour of $u$ in $G$.
If $G$ is a connected graph, it is well known that the lazy random walk over $G$ is ergodic and its distribution converges to the (unique) stationary distribution $\pi_G$ (see, e.g.,~\cite{LPW09} for a comprehensive review of random walks and mixing times).

The total variation distance $d_{\mathrm{TV}}(\mu,\nu)$ between two probability distributions $\nu$ and $\mu$ on the vertex set $V(G)$ of a graph $G$ is defined as 
\begin{equation}\label{equa:dTVdef}
d_{\mathrm{TV}}(\mu,\nu)\coloneqq\max_{A\subseteq V(G)}|\mu(A)-\nu(A)|=\frac12\sum_{v\in V(G)}|\mu(v)-\nu(v)|.
\end{equation}
Let $P_G$ be the transition matrix of the lazy random walk over $G$.
For $\eps>0$, the \emph{$\eps$-mixing time} $t_{\mix}(G,\eps)$ of this lazy random walk is defined as 
\begin{equation*}\label{equa:tmix_def}
t_{\mix}(G,\eps)\coloneqq \min\left\{t\in\mathbb{N}_0: \max_{u\in V(G)}d_{\mathrm{TV}}(\mu_0^uP_G^t,\pi_G)\leq\eps\right\},
\end{equation*}
where $\mu_0^u$ is the distribution supported entirely on $u\in V(G)$.
%As is usual, we define the \emph{mixing time} of the random walk as $t_{\mix}(G)\coloneqq t_{\mix}(G,1/4)$.

If instead of considering the worst-case initial vertex we consider a uniformly random vertex $v\in V(G)$, then the quantity $d_{\mathrm{TV}}(\mu_0^v P_G^t,\pi_G)$ is a random variable. We define the  \emph{average $\eps$-mixing time} $\bar t_{\mix}(G,\eps)$ of the lazy random walk, to be the time at which the expectation of this random variable falls below the $\eps$. That is, 
\begin{equation*}\label{equa:avg_tmix_def}
\bar t_{\mix}(G,\eps)\coloneqq\min\left\{t\in\mathbb{N}_0: \frac{1}{n}\sum_{u\in V(G)} d_{\mathrm{TV}}(\mu_0^u P_G^t,\pi_G)\leq\eps\right\}.
\end{equation*}

\begin{remark}\label{rem:other_avg_tmix}
In this work, we will focus on the quantity $\bar t_{\mix}(G,\eps)$, which we believe is a natural candidate for tracking mixing times starting from a uniform vertex. Nonetheless, other related quantities have been used to measure the mixing time from a uniform starting point. 

Indeed, for a vertex $u\in V(G)$, define
\begin{equation*}
t_{\mix}^{(u)}(G,\eps)\coloneqq \min\left\{t\in\mathbb{N}_0: d_{\mathrm{TV}}(\mu_0^uP_G^t,\pi_G)\leq\eps\right\}
\end{equation*}
and consider the random variable $t_{\mix}^{(U_n)}=t_{\mix}^{(U_n)}(G,\eps)$, where $U_n$ is a vertex chosen uniformly at random from $V(G)$.
This notion was the one studied by Berestycki, Lubetzky, Peres and Sly~\cite{BLPS18}. 
It is natural to compare $\bar t_{\mix}$ to $\mathbb{E}(t_{\mix}^{(U_n)})$: in the first case, we average the total variation distance over starting vertices and take the smallest time $t$ when this average is smaller than $\eps$; in the second one, we average the mixing times over the starting vertices (see Figure~\ref{fig:1}).
 In general as functions, neither of  these notions is stronger than the other, in that one can design examples of trajectories for total variation distances $d_{\mathrm{TV}}(\mu_0^uP_G^t,\pi_G)$ for different vertices $u$, showing that $\bar t_{\mix}$ cannot be bounded by a function of $\mathbb{E}(t_{\mix}^{(U_n)})$ and vice versa. 
However, bounding either $\mathbb{E}(t_{\mix}^{(U_n)})$ or $\bar t_{\mix}$ implies that $t_{\mix}^{(U_n)}$ is small with high probability.
In the first case this is a direct application of Markov's inequality. In the second one, define $d_u(t)\coloneqq d_{\mathrm{TV}}(\mu_0^u P_G^t,\pi_G)$, for a vertex $u$, then $\bar t_{\mix}(G,\eps)$ is the time $t$ at which the expected value of   $d_u(t)$ (averaged over starting points) is less than $\eps$. By Markov's inequality, $d_{U_n}(\bar t_{\mix}(G,\eps^2))\leq \eps $ with probability at least $1-\eps$.

%Moreover, by Markov's inequality, showing that $\mathbb{E}(t_{\mix}^{(U_n)})$ is small implies that with high probability $t_{\mix}^{(U_n)}$ is small as well. 
%Similary, if 

\begin{figure}[ht]
	\begin{center}
		\includegraphics[width=0.7\linewidth]{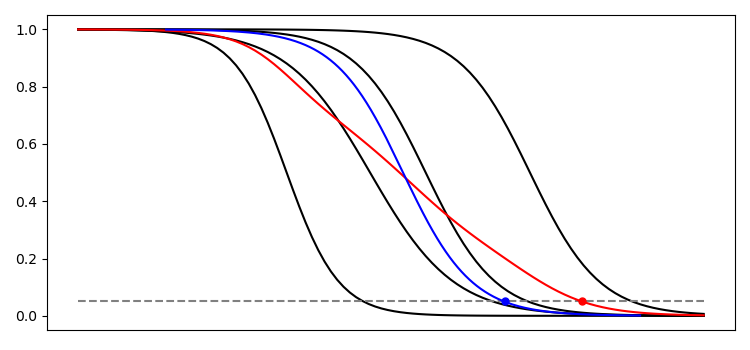}
		\caption{Schematic plot of the total variation distance starting at different vertices and the two average mixing times for $\eps= 0.05$.
        In red, the function $\frac{1}{n}\sum_{u\in V(G)} d_{\mathrm{TV}}(\mu_0^u P_G^t,\pi_G)$ and the dot representing $\bar t_{\mix}(G,\eps)$.
        In blue, the average of mixing times at different thresholds and the dot representing $\mathbb{E}(t_{\mix}^{(U_n)}(G,\eps))$.}
		\label{fig:1}
	\end{center}
\end{figure}

A related but very different  notion is the time it takes to mix starting at $\mu_V$, the uniform distribution over $V$:
\begin{equation*}
t^{(V)}_{\mix}(G,\eps)\coloneqq \min\left\{t\in\mathbb{N}_0: d_{\mathrm{TV}}(\mu_V P_G^t,\pi_G)\leq\eps\right\}.
\end{equation*}
%This quantity has been studied in the context of  Glauber dynamics in the Ising model on a torus~\cite{GS22}.
A similar notion has  been studied for directed graphs, where the initial distribution is the in-degree one; see, e.g.,~\cite[Theorem~3]{BCS18}. In general, this latter notion of average mixing time is much smaller than the previous notions and we expect this to also be the case in the settings studied here, although we do not explore this direction. 
\end{remark}

\begin{remark}\label{rem:contractive}
In the literature, the \emph{mixing time} of the random walk is often defined as $t_{\mix}(G)\coloneqq t_{\mix}(G,1/4)$, since the distance to the stationary distribution is contractive after this time. 
However, this might not be the case for $\bar t_{\mix}$.
Consider for instance the lollipop graph $L_{n,k}$: a clique on $k$ vertices and a path on $n-k$ vertices joined by an edge incident to one of the endpoints of the path.
If $k$ and $n-k$ are both very large, then, after one step, the total variation distance is roughly~$0$ if we start at the clique (almost all the mass of $\pi_{L_{n,k}}$ is supported on the clique), and roughly~$1$ if we start at the path.
Taking $k=\lceil\alpha n\rceil$, then 
\begin{equation}
    \frac{1}{n}\sum_{u\in V(G)} d_{\mathrm{TV}}(\mu_0^u P_{L_{n,k}},\pi_{L_{n,k}}) \sim  1-\alpha.
\end{equation}
If $\alpha>3/4$, then $\bar t_{\mix}(L_{n,k},1/4)=1$.
However, the time required to further decrease the distance to the stationary distribution is of order $\Omega(n^2)$, as this is the time required for the walk starting at a typical vertex in the path to hit the clique. 
\end{remark}

%We define the \emph{mixing time} of the random walk as $t_{\mix}(G)\coloneqq t_{\mix}(G,1/4)$, and its \emph{average mixing time} as $\bar t_{\mix}(G)\coloneqq \bar t_{\mix}(G,1/4)$.

\subsection{Our results}

Our results will apply to graphs satisfying certain natural structural conditions, which we formalise in the following definition.

\begin{definition} \label{def:surpressed}
Let $G$ be an $n$-vertex graph.
For $\alpha>0$, we say that a set $S\subseteq V(G)$ is \emph{$\alpha$-thin} in $G$ if 
\[|\del_G(S)|\coloneqq e_G(S,V(G)\setminus S)<\alpha |S|.\]
For $D>0$, we say that a set $S\subseteq V(G)$ is \emph{$D$-loaded} in $G$ if 
\[e_G(S)>D|S|.\]
We say that $G$ is an \emph{$(\alpha,D)$-spreader graph} if it satisfies the following three properties:
\begin{enumerate}[label = $(\mathrm{S}\arabic*)$]
\item \label{item:few suppressed} For all $(\log n)^{1/5}\le k \le (1-1/D^2) n$, 
the number of\/ $G$-connected $\alpha$-thin sets $S\subseteq V(G)$ with $|S|=k$ is less than $n\nume^{-\sqrt{k}}$.
\item \label{item:few loaded} For all $(\log n)^{1/5}\le k \le (1-1/D^2) n$, the number of\/ $G$-connected $\alpha^{-1}$-loaded sets $S\subseteq V(G)$ with $|S|=k$ is less than $n\nume^{-\sqrt{k}}$.
\item \label{item: no dense set} No set $S\subseteq V(G)$ with $|S|\geq \alpha n$ is $D$-loaded in $G$.
\end{enumerate}
\end{definition}

Note that, for $k>\log^2n$, one has that $n\nume^{-\sqrt{k}}<1$, and thus the conditions on an $n$-vertex graph $G$ being an $(\alpha,D)$-spreader graph guarantee that there are no $G$-connected vertex subsets of size between $\log^2n$ and $(1-1/D^2) n$ that have too few edges leaving the set \ref{item:few suppressed} or too many edges contained inside the set \ref{item:few loaded}.
These pseudo-random conditions on expansion and  edge distribution arise naturally in the context of random graph models. 
Indeed, the density of a random graph within any vertex set and across any vertex partition is expected to be the same as the density of the whole graph, and concentration inequalities in conjunction with union bounds can be used to derive the non-existence of such bad connected vertex sets with high probability. 
Moreover, the conditions of Definition~\ref{def:surpressed} bound the number of bad vertex sets of size between $(\log n)^{1/5}$ and $\log^2n$, with exponential decay as one can expect from concentration inequalities on binomial random variables.
In the context of the current work, the conditions on spreader graphs will guarantee that all bottlenecks are small and they are scarce in the graph.

To digest the notion of spreader graphs, one can think of $\alpha>0$ as an arbitrarily small constant and $D$ as arbitrarily large.
The parameter $\alpha>0$ controls conditions \ref{item:few suppressed} and \ref{item:few loaded} in the sense that, as $\alpha$ shrinks, these conditions become easier to satisfy and thus the definition of spreader graphs captures more graphs. 
Similarly, the parameter $D$ controls \ref{item: no dense set} and imposes in particular that the spreader graphs are sparse with bounded average degree.
It should be noted that, due to $D$ appearing in \ref{item:few suppressed} and \ref{item:few loaded} and $\alpha$ appearing in \ref{item: no dense set}, our definition is not actually monotone in these parameters.
This is a technical subtlety that is needed in our proof to guarantee a trade-off between the conditions. 
However, in all applications, the restraints given by $D$ in \ref{item:few suppressed} and \ref{item:few loaded} and $\alpha$  in \ref{item: no dense set} are never critical, as we have very good control over the edge distribution in all linear sets. 

We also remark that the constant $1/5$ could be replaced by any constant $\zeta<1/4$.
Indeed, for sets smaller than $(\log n)^{\zeta}$, we impose no restriction.
The point is that, as we will focus on connected spreader graphs $G$, even if a small set is an extreme bottleneck, the  random walk will not get stuck there for too long before exploring the set enough to escape. 
In our proof, these bottlenecks contribute a factor $(\log n)^{4\zeta}$ (due to a connected set of size $s$ having \emph{conductance} at least $1/s^2$ and Theorem~\ref{thm:FR} giving a quadratic dependence on conductance, see Section~\ref{sec:conductance} for details), hence a choice of $\zeta<1/4$ guarantees that this contribution is negligible. It may be possible to replace this constraint of $1/4$ by $1/2$ but not beyond this.

Finally, we remark that the constants $\alpha$ and $D$ could be replaced with functions that depend on $n$ and the definition of spreader graphs could be adjusted so that our main theorem would still give bounds on average mixing times.
However, as our focus is on sparse graphs with constant average degree, we do not pursue this direction here.

\begin{remark}
The definition of $(\alpha,D)$-spreader graphs bears resemblance with that of \emph{$\alpha$-AN graphs} (or \emph{$\alpha$-decorated expanders}) introduced in~\cite{BKW14}.
An $\alpha$-AN graph $G$ is defined in terms of the existence of an expander subgraph $B$ whose complement is formed by a small number of small components, similar to what can be deduced from \ref{item:few suppressed}-\ref{item: no dense set}, and additionally requiring that not too many components of $G-B$ are connected to each $v\in V(B)$.
The backbone of the main result in~\cite{BKW14} is to show that random walks on $\alpha$-AN graphs  mix in $O(\log^2 n)$ steps.
\end{remark}

Our main theorem provides a tool to prove logarithmic average mixing time for $(\alpha,D)$-spreader graphs.

\begin{theorem}\label{thm:main}
For all $\eps>0$, $D\geq 4$ and $0<\alpha<1/D^2$, there exists a $C>0$ such that the following holds for all $n$ sufficiently large.
Suppose $G$ is an $n$-vertex connected $(\alpha,D)$-spreader graph.
Then, 
\[\bar t_{\mix}(G,\eps)\leq C \log n.
\]
\end{theorem}

We believe that in many cases, as in our two applications below, this theorem can be used to quickly derive optimal bounds for average mixing times in settings where worst-case mixing times are established via conductance bounds.

The proof of Theorem~\ref{thm:main} bears some similarities with the proof in~\cite{BLPS18}.
Both use the idea of contracting badly connected sets and coupling the random walks in the original and the contracted graphs.
However, our proof is conceptually simpler as it does not use the anatomy of the giant component~\cite{DLP14}, a powerful description of the largest component in the supercritical regime.
Instead, we rely on the Fountoulakis-Reed bound for mixing~\cite{FR07} and recent progress on hitting time lemmas~\cite{MQS21}.

\subsection{Application 1: Smoothed analysis on connected graphs} \label{sec:app 1}

The idea of studying the effect of random perturbations on a given structure arose naturally in several distinct settings.
In theoretical computer science, Spielman and Teng~\cite{ST04} (see also \cite{ST09}) introduced the notion of \emph{smoothed analysis} of algorithms.
By randomly perturbing an input to an algorithm, they could interpolate between a worst-time case analysis and an average case analysis, leading to a better understanding of the practical performance of algorithms on real life instances.
This has been hugely influential, leading to the study of smoothed analysis in a host of different settings, including numerical analysis~\cite{SST06,TV07}, satisfiability~\cite{F07,CFFK09}, data clustering~\cite{AMR11}, multilinear algebra~\cite{BCMV14} and machine learning~\cite{KSS09}.
Almost simultaneously, in graph theory, Bohman, Frieze and Martin~\cite{BFM03} introduced the model of randomly perturbed graphs which, as with smoothed analysis, allows one to understand the interplay between an extremal and probabilistic viewpoint.
The majority of work on the subject has focused on dense graphs~\cite{BHKMPP19,BMPP20,HMT21}.

In the context of random walk mixing, it can be seen that small random perturbations cannot speed up the mixing time on dense graphs significantly.
Indeed, the canonical examples leading to torpid mixing (e.g., two cliques connected by a long path) are robust with respect to that property. 
Smoothed analysis of sparse graphs was introduced by Krivelevich, Reichman and Samotij~\cite{KRS15}.
Here one starts with a connected graph of bounded degree (in fact, bounded degeneracy often suffices) and applies a small random perturbation by adding a copy of the binomial random graph $R\sim G(n,{\delta}/{n})$ for small $\delta>0$.
Although this perturbation is very slight, they showed that it greatly improves the expansion properties of the graph.
A graph $G$ is said to be $\Delta$-degenerate if there is some ordering of the vertices of $G$ such that each vertex has at most $\Delta$ neighbours in $G$ that precede it in the ordering.
To be precise, Krivelevich, Reichman and Samotij proved that, for any $\Delta\in \NN$ and $\delta>0$, if\/ $G$ is an $n$-vertex $\Delta$-degenerate connected graph and $R\sim G(n,{\delta}/{n})$, then $G'\coloneqq G\cup R$ a.a.s.\ satisfies that $t_\mix(G')=O(\log ^2n)$.
By considering, for example, a path on $n$ vertices, which has mixing time $\Omega(n^2)$, we see a vast improvement after a slight random perturbation.
We also note that the result is tight on such examples, as the randomly perturbed path a.a.s.\ contains bare paths of length $\Omega(\log n)$.

Our first application of Theorem~\ref{thm:main} shows that we can improve the mixing time yet further in this model by starting from a uniformly chosen vertex, as in this case we avoid the small bottlenecks that remain if the initial graph had poor expansion. 

\begin{theorem} \label{thm:smoothed}
For any\/ $\eps,\delta>0$ and\/ $\Delta\in \NN$, there exists a\/ $C>0$ such that the following holds.
Let\/ $G$ be an\/ $n$-vertex\/ $\Delta$-degenerate connected graph, choose\/ $R\sim G(n,{\delta}/{n})$ and let\/ $G'\coloneqq G\cup R$.
Then, a.a.s.
\[\bar t_{\mix}(G',\eps)\leq C\log n.\]
\end{theorem}

\begin{remark}\label{rem:tightness}
Theorem~\ref{thm:smoothed} is tight, up to the constant factor $C$, for \emph{all} graphs with maximum degree\/ $\Delta$.
Indeed, this follows from the fact that $N^k(v)=O( (2\bar d)^k\log n)$ for \emph{all} vertices $v\in V(G')$, where $N^k(v)$ denotes the number of vertices that are at distance at most $k$ from $v$ in $G'$ and $\bar d\coloneqq\Delta +\delta$ is an upper bound on the average degree in $G'$.
Such an upper bound can be shown easily by induction, see for example~\cite{CL01}, and setting $k=c\log n$ for $c>0$ sufficiently small shows that at least half of the vertices cannot be reached from $v$ in $k$ steps and hence $\bar t_\mix(G',\eps)\geq k$.
Nonetheless, the converse of the inequality in Theorem~\ref{thm:smoothed} is not true for all\/ $\Delta$-degenerate graphs.
Consider for instance a star: it is\/ $1$-degenerate, but the mixing time of the randomly perturbed star is\/ $O(1)$ as we mix in the step after visiting the centre of the star for the first time.
\end{remark}
 
Some time before the systematic study of random perturbations in the combinatorial and theoretical computer science communities discussed above, the notion appeared in physics literature with the study of so-called \emph{small-world networks}.
Here we will concentrate on a model introduced by Newman and Watts~\cite{NW99a,NW99b} where, for some fixed $k\in \NN$, $\delta>0$ and $n\in \NN$ large, one starts with $n$-vertices of the graph ordered as $v_1,\ldots, v_n$, adds all edges $v_iv_j$ for which $i+1\leq j \leq i+k$ (with addition modulo $n$), and then adds all remaining edges independently with probability $p=\delta/n$.
We denote the resulting random graph as $H_{n,k,\delta}$. 
It is easy to see that this graph fits into the framework of Krivelevich, Reichman and Samotij~\cite{KRS15}, and so their result implies that, for any $k\in \NN$ and $\delta>0$, the Newman-Watts small world network $H_{n,k,\delta}$ a.a.s.\ satisfies $t_\mix(H_{n,k,\delta})=O(\log ^2n)$.
In fact, this was established before their work by Addario-Berry and Lei~\cite{ABL15}, improving on a previous bound of $O(\log ^3n)$ due to Durrett~\cite{D07}.
Here, as a direct consequence of Theorem~\ref{thm:smoothed}, we conclude that the average mixing time on the Newman-Watts small world network is of order $O(\log n)$. 

\begin{corollary} \label{cor:NW small world}
For all\/ $k\in \NN$ and\/ $\eps,\delta>0$, there exists a\/ $C>0$ such that the following holds.
The Newman-Watts small world network\/ $H_{n,k,\delta}$ a.a.s.\ satisfies 
\[\bar t_\mix(H_{n,k,\delta},\eps)\leq C\log n. \] 
\end{corollary}

\subsection{Application 2: Giant components in random subgraphs of expanders} \label{sec:app 2}

For $p\in [0,1]$ and a graph $G$, we define $G_p$ to be the graph with the same vertex set where each edge of $G$ is retained in $G_p$ independently with probability $p$.
The graph $G$ is called the host graph, and the random subgraph $G_p$, the $p$-percolated one.
Percolation on graphs is a well-established topic in probability theory.
Most classically, if the host graph is the complete graph on $n$ vertices $K_n$, then its $p$-percolated subgraph is the Erd\H os-R\'enyi graph $G(n,p)$.
For any graph $G$, let $L_1(G)$ denote a largest connected component in $G$ and let $\ell_1(G)$ denote its order. 
In their seminal paper~\cite{ER60}, Erd\H os and R\'enyi proved a phase transition for $\ell_1(G(n,p))$.
Namely, writing $p=c/n$ for some constant $c$, if $c<1$ then a.a.s.\ $\ell_1(G(n,p))=O(\log n)$, while if $c>1$ then a.a.s.\ $\ell_1(G(n,p))=\Omega(n)$ and the largest component, which is the unique component of linear size, is known as the \emph{giant component}.

A central question in random graph theory is whether other host graphs exhibit the same phenomenon~\cite{AKS82}.
One quickly observes that, in order for $G_p$ to have a sharp threshold for the component structure, the host graph $G$ should satisfy some additional properties.
A natural property to consider is the pseudo-random notion of \emph{expansion}.
There is a strong connection between expansion and the graph spectrum.
Given the eigenvalues of the adjacency matrix of a $d$-regular graph $G$, say $d=\lambda_1\geq \lambda_2\geq \ldots \geq \lambda _n$, we let $\lambda(G)\coloneqq\max\{|\lambda_2|,|\lambda_n|\}$ be the second largest eigenvalue.
We then define an \emph{$(n, d, \lambda)$-graph} to be a $d$-regular graph $G$ on $n$ vertices with $\lambda(G)=\lambda$.
When $\lambda$ is small compared to $d$, an $(n,d,\lambda)$-graph is said to be an \emph{expander} and it enjoys many of the same properties as a random graph with the same density.
We refer the reader to the excellent survey of Krivelevich and Sudakov~\cite{KS06} on the subject. 

In terms of percolation, Frieze, Krivelevich and Martin~\cite{FKM04} proved that, if $G$ is an $(n,d,\lambda)$-graph with $\lambda=o(d)$, then $\ell_1(G_p)$ undergoes a phase transition at $p=1/d$.
They obtained the following description of the supercritical regime: for $\delta>0$ and $p=(1+\delta)/d$, and provided that $\lambda\leq\delta^4d$, a.a.s.\ $\ell_1(G_p)\sim y(\delta)n$ for some $y(\delta)\in (0,1)$.
Moreover, as in $G(n,p)$, the largest component $L_1(G_p)$ is a.a.s.\ the unique component of linear size.
Very recently, Diskin and Krivelevich~\cite{DK22} studied the mixing time of percolated $(n, d, \lambda)$-graphs.
More precisely, they showed that, in the supercritical regime, there exists $C=C(\delta)$ such that a.a.s.\ $t_{\mix}(L_1(G_p))\leq C \log ^2n$.
This is indeed optimal for some graphs, in particular for Erd\H os-R\'enyi random graphs~\cite{FR08,BKW14}, as discussed above.

Our next application of our main result shows that for percolated pseudo-random graphs the average mixing time is logarithmic.

\begin{theorem} \label{thm:expanders}
For all $\delta>0$ sufficiently small and all $\eps>0$, 
there exists a $C>0$ such that, if $p=(1+\delta)/{d}$ and $G$ is an $(n,d,\lambda)$-graph with $\lambda\leq\delta^4d$, then a.a.s.
\[\bar t_{\mix}(L_1(G_p),\eps)\le C \log n.\]
\end{theorem}
Similarly as in Remark~\ref{rem:tightness}, it can be proven that Theorem~\ref{thm:expanders} is tight up to multiplicative constant for \emph{all} $(n,d,\lambda)$-graphs.

As a consequence, for $G=K_n$ we obtain the following.

\begin{corollary} \label{cor:randomgraph}
For all $\delta>0$ sufficiently small and all $\eps>0$, there exists a $C>0$ such that, for $p=(1+\delta)/n$, a.a.s.
\[\bar t_{\mix}(L_1(G(n,p)),\eps)\le C \log n.\]
\end{corollary}

By Remark~\ref{rem:other_avg_tmix}, $t^{(U_n)}_{\mix}=O(\log{n})$ a.a.s., where $U_n$ is chosen uniformly at random from $V(G)$. This result aligns with~\cite{BLPS18}, although theirs  is much stronger, showing cut-off for $t^{(U_n)}_{\mix}$ as previously mentioned. 
%who studied average mixing times on the Erd\H os-R\'enyi random graph in the sense of random starting vertices (see Remark~\ref{rem:other_avg_tmix}). In particular, they showed something stronger: there is some $b=b(\delta)$ such that, for all $\eps\in (0,1)$, $t_{\mix}^{(u)}(L_1(G(n,p)),\eps)\sim b \log{n}$, which implies that the random walk starting at a given vertex has cut-off.

\subsection{Organisation} \label{sec:organisation}

The rest of this paper is organised as follows.
In \cref{sec:prelims} we introduce all the necessary notation, definitions and tools for our proofs.
We use these to prove Theorem~\ref{thm:main} in \cref{sec:main}.
This section is structured in subsections where we build different tools to be used in our main proof; in particular, we discuss the main ideas of the proof in \cref{sec:overview}.
\Cref{sec:smoothed,sec:expanders} are devoted to proving Theorems~\ref{thm:smoothed} and \ref{thm:expanders}, respectively.
Finally, we discuss some open problems in \cref{sec:conc}.

\section{Preliminaries} \label{sec:prelims}

\subsection{Basic notation}
\label{sec:notation}
 
Let $\mathbb{N}_0=\{0,1,2,\ldots\}$ denote the set of non-negative integers.
If $n$ is a positive integer, we set $[n]\coloneqq\{1,\ldots,n\}$.
Throughout, we will consider both simple graphs and multigraphs.
The word \emph{graph} will refer to simple graphs, that is, each pair of vertices forms at most one edge.
Our multigraphs, which will be allowed to have parallel edges but no loops, will be clearly identified as such.
All our graphs are labelled, so whenever we discuss an $n$-vertex (multi)graph $G$, we implicitly assume that $V(G)=[n]$.
Given a (multi)graph $G=(V,E)$ and disjoint sets $A,B\subseteq V(G)$, we write $E_G(A)$ for the (multi)set of edges of $G$ contained in $A$, and $E_G(A,B)$ for the (multi)set of edges with one endpoint in $A$ and the other in $B$.
We set $e_G(A)\coloneqq|E_G(A)|$ and $e_G(A,B)\coloneqq|E_G(A,B)|$.
If $A=\{a\}$, we write $E_G(a,B)\coloneqq E_G(\{a\},B)$, and similarly in all related notation.
For simplicity, we write $e(G)\coloneqq e_G(V(G))$.
We write $G[A]\coloneqq(A,E_G(A))$.
We say that $A$ is \emph{$G$-connected} if $G[A]$ is connected.
For each vertex $v\in V(G)$, we write $\deg_G(v)\coloneqq e_G(\{v\},V(G)\setminus\{v\})$ for its degree.
We denote $\partial_G(A)\coloneqq E_G(A,V(G)\setminus A)$ and $\deg_G(A)\coloneqq\sum_{v\in A}\deg_G(v)=2e_G(A)+|\del_G(A)|$.

In many of our statements we will consider an $n$-vertex graph satisfying a set of conditions or a conclusion, which are often asymptotic in nature.
This is in fact an abuse of notation.
To be precise, one must consider a sequence $(G_k)_{k\geq 1}$ of graphs on an increasing number of vertices so that the graphs in the sequence satisfy the conditions.
This abuse of notation greatly simplifies the statements, so we will assume it throughout.
(This also includes any asymptotic statements about random graphs.)
For any sequence of graphs $(G_k)_{k\geq 1}$ with $|V(G_k)|\to \infty$, we say that a graph property $\mathcal{P}$ holds asymptotically almost surely (a.a.s.) if $\lim_{k\to\infty}\mathbb{P}[G_k\in\mathcal{P}]=1$.

\subsection{Random walks}\label{sec:conductance}

Given an arbitrary connected multigraph $G$, the lazy random walk over $G$ is a Markov chain on state space $V(G)$ defined by the transition matrix $P_G=(P_G(i,j))_{i,j\in V(G)}$ given by
\begin{equation}\label{equa:randwalkdef}
    P_G(i,j)=\begin{cases}
1/2&\text{if }i=j,\\
e_G(i,j)/(2\deg_G(i))&\text{if }i\neq j.
\end{cases}
\end{equation}
That is, the lazy random walk is a sequence of random variables $(X_t)_{t\geq0}$ with probability distributions $(\mu_t)_{t\geq0}$, respectively, over $V(G)$, where $\mu_0$ is the starting distribution and, for each $t\geq1$, the distribution of $\mu_t$ is obtained from the distribution of $\mu_{t-1}$ as $\mu_t=\mu_{t-1}P_G=\mu_0P_G^t$.
The sequence of distributions thus depends only on $G$ and the starting distribution.
In the special case when there is a vertex $x\in V(G)$ such that $\mu_0(x)=1$, we will write $(\mu_t^x)_{t\geq0}$ to denote the resulting sequence of distributions.

If $G$ is connected, the lazy random walk over $G$ converges to a stationary distribution $\pi_G$ (that is, a distribution satisfying $\pi_G=\pi_GP_G$), independently of the starting distribution $\mu_0$. 
It is well known (see, e.g., \cite{LPW09}) that this stationary distribution satisfies 
\begin{equation}\label{equa:statdistdef1}
\pi_G(u)=\frac{\deg_G(u)}{2e(G)}
\end{equation}
for all $u\in V(G)$.
Given a set $S\subseteq V(G)$, we define $\pi_G(S)\coloneqq\sum_{v\in S}\pi_G(v)$.
It follows from \eqref{equa:statdistdef1} that 
\begin{equation}\label{equa:statdistdef2}
\pi_G(S)=\frac{\deg_G(S)}{2e(G)}=\frac{2e_G(S)+|\partial_G(S)|}{2e(G)}.
\end{equation}
We define $\pi_{\min}(G)\coloneqq\min_{v\in V(G)}\pi_G(v)$ and $\pi_{\max}(G)\coloneqq\max_{v\in V(G)}\pi_G(v)$.

Recall the definition of mixing times in the introduction.
The mixing time of a random walk on a connected (multi)graph $G$ is deeply tied with the concept of \emph{conductance}.
Given a set $S\subseteq V(G)$, we define
\begin{equation}\label{equa:conductdef1}
Q_G(S)\coloneqq\sum_{i\in S}\sum_{j\in V(G)\setminus S}\pi_G(i)P_G(i,j)=\frac{|\partial_G(S)|}{4e(G)},
\end{equation}
where the equality follows from \eqref{equa:randwalkdef} and \eqref{equa:statdistdef1}.
Observe that $Q_G(S)=Q_G(V(G)\setminus S)$.
Finally, we define the \emph{conductance} $\Phi_G(S)$ of $S$ as
\begin{equation}\label{equa:conductdef2}
\Phi_G(S)\coloneqq\frac{Q_G(S)}{\pi_G(S)\pi_G(V(G)\setminus S)}.
\end{equation}
From the definitions in \eqref{equa:statdistdef2}, \eqref{equa:conductdef1} and \eqref{equa:conductdef2} and the fact that $\deg_G(A)\leq2e(G)$ for any set $A\subseteq V(G)$, it follows that 
\begin{equation} \label{eq:conductance simplified}
\Phi_G(S)=\frac{e(G)|\partial_G(S)|}{\deg_G(S)\deg_G(V(G)\setminus S)}\geq\frac{|\partial_G(S)|}{2\deg_G(S)}.
\end{equation}

Our approach to estimate the mixing time of the lazy random walk over a multigraph $G$ is based on ideas of Fountoulakis and Reed~\cite{FR07,FR08}.
Roughly speaking, their main contribution is the fact that the mixing time of an abstract irreducible, reversible, aperiodic Markov chain (which we may represent using a weighted graph $H$ on its state space) can be bounded from above using the conductances of different $H$-connected sets of states of various sizes. 
The fact that we may restrict ourselves to $H$-connected sets is crucial to obtain tighter bounds than would be obtained through other classical means.
For simplicity, here we only state a version of the result of Fountoulakis and Reed~\cite{FR07} which is applicable to our setting.
For any $p\in(\pi_{\min}(G),1)$, we let $\Phi_G(p)$ be the minimum conductance $\Phi_G(S)$ over all $G$-connected sets $S\subseteq V(G)$ such that $p/2\leq\pi_G(S)\leq p$ (if no such set $S$ exists, we set $\Phi_G(p)=1$).

\begin{theorem}[Fountoulakis and Reed~\cite{FR07}]\label{thm:FR}
Let\/ $G$ be a connected multigraph.
There exists an absolute constant\/ $C_0$ such that
\[t_{\mathrm{mix}}(G)\leq C_0\sum_{j=1}^{\lceil\log_2\pi_{\min}(G)^{-1}\rceil}\Phi_G^{-2}(2^{-j}).\]
\end{theorem}

Another parameter of interest is the \emph{hitting time} to a vertex (or set of vertices) in the random walk on a multigraph $G$.
Given any $v\in V(G)$ and the lazy random walk $(X_t)_{t\geq0}$ with starting distribution $\mu_0$, we define the \emph{hitting time to $v$} as
\[\tau_G(\mu_0,v)\coloneqq\inf\{t\in\mathbb{N}_0:X_t=v, X_0\sim\mu_0\}.\]
In more generality, given any set $S\subseteq V(G)$, we define the \emph{hitting time} to $S$ as
\[\tau_G(\mu_0,S)\coloneqq\inf\{t\in\mathbb{N}_0:X_t\in S, X_0\sim\mu_0\}.\]

Given any vertex $u\in V(G)$, let $P_u$ be the matrix obtained from the transition matrix $P_G$ by removing the row and column corresponding to $u$.
If $P_u$ is primitive (i.e., all entries of $P_u^m$ are positive for some $m\geq 1$), by Perron-Frobenius, the largest eigenvalue of $P_u$, denoted by $\lambda_u$, is real, of multiplicity $1$ and satisfies $\lambda_u<1$.

We will make use of the first visit time lemma of Cooper and Frieze~\cite{CF05}.
Here we state a more recent version with weaker hypotheses due to Manzo, Quattropani and Scoppola~\cite{MQS21}.

\begin{theorem}[First Visit Time Lemma, Manzo, Quattropani and Scoppola~\cite{MQS21}]\label{thm:MQS}
Let\/ $G$ be an\/ $n$-vertex connected multigraph.
Suppose that there exist a real number\/ $c>2$ and a diverging sequence\/ $T=T(n)$ such that the following conditions hold:
\begin{enumerate}[label=$(\mathrm{HP}\arabic*)$]
    \item\label{item:HP1} Fast mixing:\/ $\max_{x,y\in V(G)}|\mu_T^x(y)-\pi_G(y)|=o(n^{-c})$.
    \item\label{item:HP2} Small\/ $\pi_{\max}$:\/ $T\cdot\pi_{\max}(G)=o(1)$.
    \item\label{item:HP3} Large\/ $\pi_{\min}$:\/ $\pi_{\min}(G)=\omega(n^{-2})$.
\end{enumerate}
Then, for all\/ $u\in V(G)$, we have
\begin{equation}\label{eq:hitting_geom}
    \sup_{t\geq0}\left\lvert\frac{\mathbb{P}[\tau_G(\pi_G,u)>t]}{\lambda_u^t}-1\right\rvert\xrightarrow[n\to\infty]{}0
\end{equation}
and
\begin{equation}\label{eq:prob_geom}
    \left\lvert\frac{1-\lambda_u}{\pi_G(u)/R_T(u)}-1\right\rvert\xrightarrow[n\to\infty]{}0,
\end{equation}
where
\[R_T(u)\coloneqq\sum_{t=0}^T\mu_t^u(u)>1\]
is the expected number of indices\/ $t\in[T]\cup\{0\}$ for which the lazy random walk\/ $(X_t)_{t\geq0}$ on\/ $G$ starting at\/ $X_0=u$ satisfies\/ $X_t=u$.
\end{theorem}

From the intuitive point of view, the theorem says that the hitting time to $u$ is roughly distributed as a geometric random variable with success probability $\pi_G(u)/R_T(u)$.
If one wants to hit $u$ by independently sampling vertices according to $\pi_G$, then it would be a geometric random variable with success probability $\pi_G(u)$.
The factor $R_T(u)$ is the price to pay for taking into account the geometry of the graph: the more likely it is to return from $u$ to $u$, the less connected $u$ is to the rest of the graph, and the smaller the probability to hit it at a given (large) time is.

%The conclusion in \eqref{eq:hitting_geom} is slightly different from the one in \cite[Theorem 2.1]{MQS21}: we replaced the starting distribution $\pi_G$ by $\mu_0^v$, which holds as remarked in \cite[Section 2.2, Observation (2)]{MQS21}.

\begin{remark}\label{rem:pi_max_not_needed}
In the proof of Theorem~\ref{thm:MQS}, one can check that, if we only want \eqref{eq:hitting_geom} to hold for a given\/ $u\in V(G)$, then \ref{item:HP2} can be replaced by
\begin{enumerate}[label=$(\mathrm{HP}\arabic*')$]
    \setcounter{enumi}{1}
    \item\label{item:HP2'} Small\/ $\pi_G(u)$:\/ $T\cdot\pi_G(u)=o(1)$.
\end{enumerate}
\end{remark}

\begin{remark}\label{rem:replace pi}
The following holds as a corollary of Theorem~\ref{thm:MQS}.
For any fixed $D>0$ and $n$-vertex connected multigraph $G$ satisfying \ref{item:HP1}, \ref{item:HP2} (or~\ref{item:HP2'}) and~\ref{item:HP3} with the additional property that $e(G)\leq Dn$, if $u\in V(G)$ and $t_0=t_0(n)$ is such that $\lambda_u^{t_0}=1-o(1)$, then 
\begin{equation} \label{eq: replace pi eq}
\frac{1}{n}\sum_{v\in V(G)} \mathbb{P}[\tau_G(\mu_0^v,u)\leq t_0]=o(1).
\end{equation}
Indeed, for $\eps>0$, Theorem~\ref{thm:MQS} implies that $\mathbb{P}[\tau_G(\pi_G,u)\leq t_0]\leq \eps^2$. 
Then, $B_{\eps}\coloneqq\{v\in V(G):\mathbb{P}[\tau_G(\mu_0^v,u)\leq t_0]\geq \eps\}$ satisfies $\pi_G(B)\leq \eps$.
Moreover, as $G$ is connected, we have that $\pi_G(B)\geq |B|/(2e(G))\geq |B|/(2Dn)$  and so $|B|\leq 2\eps D n$.
Therefore, 
\begin{align*}
    \sum_{v\in V(G)}\mathbb{P}[\tau_G(\mu_0^v,u)\leq t_0]&=\sum_{v\in V(G)\setminus  B}\mathbb{P}[\tau_G(\mu_0^v,u)\leq t_0]+\sum_{v\in B}\mathbb{P}[\tau_G(\mu_0^v,u)\leq t_0]\\
    &\leq (2D+1)\eps n,
\end{align*}
which can be made arbitrarily small, by taking $\eps$ small with respect to $D$.
This establishes \eqref{eq: replace pi eq}.
\end{remark}

\section{A general approach to average mixing times} \label{sec:main}

\subsection{Proof overview} \label{sec:overview}

As discussed in the introduction, the main tool we will use to bound the mixing times of random walks is the result of Fountoulakis and Reed (Theorem~\ref{thm:FR}) which relates the (worst-case) mixing time of a random walk in a graph $G$ to the conductance (see \eqref{equa:conductdef2}) of the $G$-connected vertex subsets $S$ of $G$.
We think of vertex subsets $S$ whose conductance is \emph{poor} (those for which $\Phi_G(S)=o(1)$) as bottlenecks: they have more edges internally in $S$ than leaving $S$ and so the random walk is likely to get held up in $S$.
The spreader graphs (see Definition~\ref{def:surpressed}) we are interested in studying here have only few small bottlenecks.
Indeed, any vertex subset which can lead to small bottlenecks must either be thin or loaded and our upper bounds on the number of these sets in a spreader graph readily imply that any vertex set with poor conductance is of at most polylogarithmic size (size $(\log^2n)$ to be precise, see Remark~\ref{rem:no big bad}).
Now, if a set with poor conductance is very small (size at most $(\log n)^{1/5}$), it will not slow down mixing significantly as our random walk will not get stuck for very long in these sets before leaving them.
Therefore, it is the intermediate size sets which pose a problem, and we will first show in \cref{sec:bad} (see Lemma~\ref{lem:smallU}) that the set $U$ of bad vertices contained in \emph{some} intermediate set which has poor conductance contains a negligible proportion of the overall vertex set of our spreader graph. 

Intuitively, we can then see how starting at an average vertex in $G$ speeds up the mixing time.
Indeed, we are very unlikely to start at a bad vertex in $U$ and, moreover, we are in fact very unlikely to visit a vertex in $U$ in the first $O(\log n)$ time steps, by which time we aim to show that the distribution of the random walk is already well-mixed.
In order to formalise this intuition, we adjust our spreader graph $G$ by shrinking the intermediate sets with poor conductance and thus removing troublesome small bottlenecks.
The resulting (multi-)graph we will call $G^*$.
Using that the number of bad vertices $|U|$ is negligible, or rather that the number of edges incident to $U$, $\deg_G(U)$, is negligible (Lemma~\ref{lem:smallU}), we show in \cref{sec:stationary} that switching from $G$ to $G^*$ does not have a big effect on the edge distribution and that, in particular, the stationary distributions on $G$ and $G^*$ are comparable.
We then show in \cref{sec:contract} that, after contracting intermediate sets with poor conductance, we can apply Theorem~\ref{thm:FR} of Fountoulakis and Reed to conclude that the \emph{worst-case} mixing time in $G^*$ is logarithmic.
Here we will need that $G^*$ is defined carefully to preserve connectivity between sets after contractions (see \eqref{eq:f}).
Finally, we will prove Theorem~\ref{thm:main} by coupling the random walk from an average vertex on $G$ with the random walk in $G^*$.
As the random walk in $G^*$ from \emph{any} starting point mixes rapidly, we can conclude that the random walk in $G$ also mixes rapidly as long as the two random walks stay coupled for long enough.
For this, our final ingredient is to show that the random walk in $G$ is unlikely to hit our bad vertices $U$, which we do in \cref{sec:hitting} by appealing to the First Visit Time Lemma (Theorem~\ref{thm:MQS}) of Manzo, Quattropani and Scoppola.

\subsection{Badly connected sets} \label{sec:bad}

We will make use of the following simple definition.

\begin{definition} \label{def:bad}
For $\gamma>0$ and a connected multigraph $G$, we say a set $S\subseteq V(G)$ is \emph{$\gamma$-bad} in $G$ if 
\[\frac{|\del_G(S)|}{\deg_G(S)}<\gamma, \]
and that it is \emph{$\gamma$-good} otherwise.
\end{definition}

The following lemma gives us a basic property of bad sets in a connected multigraph and quickly ties them with our notion of $(\alpha,D)$-spreader graphs.
Note that the notions of thin and loaded sets extend naturally to multigraphs.

\begin{lemma} \label{lem:simplebad}
Let $G$ be a multigraph.
For\/ $0<\alpha\leq1$, if a set $S\subseteq V(G)$ is $({\alpha^2}/{4})$-bad in $G$, then
\begin{enumerate}[label=$(\arabic*)$]
    \item \label{item:bad1} $|\del_G(S)|\le 2e_G(S)$,
    and
    \item \label{item:bad2} either $S$ is $\alpha$-thin in $G$ or it is $\alpha^{-1}$-loaded in $G$ (or both).
\end{enumerate}

\end{lemma}
\begin{proof}
The assertion \ref{item:bad1} follows easily since, if $|\del_G(S)| > 2e_G(S)$, then 
\[\frac{|\del_G(S)|}{\deg_G(S)}=\frac{|\del_G(S)|}{2e_G(S)+|\del_G(S)|}\geq \frac{|\del_G(S)|}{2|\del_G(S)|} =\frac{1}{2}\geq\frac{\alpha^2}{4}, \]
a contradiction.

For the second assertion, suppose that $S$ is neither $\alpha$-thin nor $\alpha^{-1}$-loaded in $G$.
Then, we have that 
\[\frac{|\del_G(S)|}{\deg_G(S)}=\frac{|\del_G(S)|}{2e_G(S)+|\del_G(S)|}\geq \frac{|\del_G(S)|}{4e_G(S)} \geq \frac{\alpha|S|}{4\alpha^{-1}|S|} = \frac{\alpha^2}{4},\]
a contradiction.
Here we used that $|\del_G(S)|\leq 2e_G(S)$ from \ref{item:bad1} in the first inequality and the definitions of $\alpha$-thin and $\alpha^{-1}$-loaded in the second. 
\end{proof}

We will also make use of the following simple observation.

\begin{remark} \label{rem:no big bad}
Let $\alpha,D>0$, and let $G$ be an $n$-vertex graph satisfying \ref{item:few suppressed} and \ref{item:few loaded}.
Then, there are \emph{no} $G$-connected $\alpha$-thin or \mbox{$\alpha^{-1}$-loaded} sets $S$ of size $(\log n)^2\leq|S|\leq(1-1/D^2)n$.
\end{remark}

Given any $\alpha>0$ and $D>0$, we define $\beta\coloneqq1/D^2$ and $\gamma\coloneqq{\alpha^2}/{4}$.
Given any $n$-vertex graph $G$, let
\begin{equation}\label{eq:cS}
\cS(\alpha,D)\coloneqq\big\{S\subseteq V(G):S \text{ is $G$-connected and } \gamma\text{-bad}, (\log n)^{1/5}\leq |S|\leq (1-\beta)n\big\}.
\end{equation}
Further, we define $U(\alpha,D)\subseteq V(G)$ to be the set of vertices that lie in sets in $\cS(\alpha,D)$, that is,
\begin{equation} \label{eq: U def}
    U(\alpha,D)\coloneqq\bigcup_{S \in \cS(\alpha,D)}S.
\end{equation}

We now give bounds on $|U(\alpha,D)|$ and $\deg_G(U(\alpha,D))$. 

\begin{lemma}\label{lem:smallU}
Let $\alpha\in (0,1]$ and $D>0$.
Let $G$ be an $n$-vertex connected graph which satisfies \ref{item:few suppressed} and \ref{item:few loaded}.
If $n$ is sufficiently large, then
\begin{equation}\label{eq:smallU1}
|U(\alpha,D)|\le n\exp(-(\log n)^{1/11})
\end{equation}
and
\begin{equation}\label{eq:smallU2}
\deg_G(U(\alpha,D))\le n\exp(-(\log n)^{1/11}).
\end{equation}
In particular,
\begin{equation}\label{eq:smallU3}
\pi_G(U(\alpha,D))\le \exp(-(\log n)^{1/11}).
\end{equation}
\end{lemma}

\begin{proof}
Let $\cS=\cS(\alpha,D)$ and $U=U(\alpha,D)$.
By Lemma~\ref{lem:simplebad}\ref{item:bad2}, all sets $S\in\mathcal{S}$ must be $\alpha$-thin or $\alpha^{-1}$-loaded.
By \ref{item:few suppressed} and \ref{item:few loaded}, for each $(\log n)^{1/5}\leq k\leq (1-\beta)n$, the number of $\alpha$-thin or $\alpha^{-1}$-loaded sets of size $k$ in $G$ is less than $2n\exp(-\sqrt{k})$.
Moreover, as stated in Remark~\ref{rem:no big bad}, there are no $\alpha$-thin or $\alpha^{-1}$-loaded sets of size $k\geq(\log n)^2$.
Thus,
\begin{equation}\label{eq:smallU4}
    |\cS|<2n\sum_{k=(\log n)^{1/5}}^{(\log n)^2}\nume^{-\sqrt{k}}\leq2n(\log n)^2\nume^{-(\log n)^{1/10}}.
\end{equation}
Therefore, by the definition of $U$, the bound on the size of the largest $S\in\cS$, and assuming $n$ is sufficiently large, we conclude that
\[|U|\leq\sum_{S \in \cS}|S|\leq|\cS|(\log n)^2<2n(\log n)^4\nume^{-(\log n)^{1/10}}\leq n\nume^{-(\log n)^{1/11}}.\]

We now turn our attention to \eqref{eq:smallU2}.
By the definition of $U$, we have that 
\[\deg_G(U)\leq \sum_{S\in \cS}\deg_G(S),\]
and using Lemma~\ref{lem:simplebad}\ref{item:bad1} this simplifies to 
\[\deg_G(U)\leq \sum_{S\in \cS}4e_G(S).\]
Now, for each $S\in \cS$, let $b(S)$ be some $G$-connected set such that $(\log n)^2\leq|b(S)|\leq(\log n)^3$ and $S\subseteq b(S)$.
Note that this is possible because $G$ is connected and every $S\in\cS$ has size less than $(\log n)^2$.
By \ref{item:few loaded} and the bounds on $|b(S)|$, for each $S\in \cS$ we have that
\[e_G(b(S))\le \alpha^{-1}|b(S)|\le \alpha^{-1}(\log n)^3.\]
Therefore, using \eqref{eq:smallU4} and for $n$ sufficiently large,
\[
\deg_G(U)\leq \sum_{S\in \cS}4e_G(S)\le \sum_{S\in \cS}4e_G(b(S))\le 4\alpha^{-1}(\log n)^3 |\cS|\le n\nume^{-(\log n)^{1/11}}.
\]
In particular, since $G$ is connected (so $e(G)\geq n/2$), it follows from \eqref{equa:statdistdef2} that
\[
\pi_G(U)=\frac{\deg_G(U)}{2e(G)}\leq \nume^{-(\log n)^{1/11}}.\qedhere
\]
\end{proof}

\subsection{Stationary distributions} \label{sec:stationary}

Let $\alpha>0$ and $D>0$ be given, let $G$ be some $n$-vertex graph, and consider the set $U=U(\alpha,D)$.
We now define $G^*=G^*(\alpha,D)$ to be the multigraph obtained by contracting all the connected components of $G[U]$ to single vertices.
To be more precise, let $U=U_1\cup\ldots\cup U_t$ be a partition of $U$ into sets each of which induces a connected component in $G[U]$ and let $U^*\coloneqq\{u_1,\ldots,u_t\}$ be a set of $t$ new vertices.
Then, let $V(G^*)=(V(G)\setminus U)\cup U^*$ and, for each $x\in V(G)$, let 
\begin{equation} \label{eq:f}
f(x)= \begin{cases}
    x & \text{if } x\notin U, \\
    u_i & \text{if } x\in U_i\subseteq U.
\end{cases}
\end{equation}
In particular, $f(U)=U^*$.
Finally, we define the multiset \[E(G^*)\coloneqq\{f(x)f(y):xy\in E(G), f(x)\neq f(y)\}.\]
Observe that $G^*$ is connected if and only if $G$ is connected, and that $U^*$ must be an independent set in $G^*$.

Given some connected graph $G$, we want to compare the behaviour of the lazy random walk on $G$ and its contracted form $G^*$.
In particular, we wish to compare their stationary distributions.
In order to do this, we need to make them comparable by having them on the same state space.
Let us describe this in full generality.
Let $H_1=(V_1,E_1)$ and $H_2=(V_2,E_2)$ be two connected multigraphs (possibly with $V_1\cap V_2\neq\varnothing$).
Then, we define an auxiliary multigraph $H\coloneqq H_1\cup H_2=(V_1\cup V_2,E_1\cup E_2)$ as the union of $H_1$ and $H_2$.
Given the stationary distributions $\pi_{H_1}$ and $\pi_{H_2}$, we define two distributions $\sigma_1$ and $\sigma_2$ on $H$, where, for each $v\in V(H)$,
\begin{equation}\label{eq:now}
    \begin{aligned}
    \sigma_1(v)&=\begin{cases}\pi_{H_1}(v)&\text{if }v\in V_1,\\0&\text{otherwise};\end{cases}\\
    \sigma_2(v)&=\begin{cases}\pi_{H_2}(v)&\text{if }v\in V_2,\\0&\text{otherwise}.\end{cases} 
\end{aligned}
\end{equation}
For the sake of notation, for any vertex $v\in V(H)\setminus V_1$ we set $\deg_{H_1}(v)=0$ and, similarly, for any $v\in V(H)\setminus V_2$ we set $\deg_{H_2}(v)=0$. 
With this setup, we abuse notation slightly and write $d_{\mathrm{TV}}(\pi_{H_1},\pi_{H_2})$ for $d_{\mathrm{TV}}(\sigma_1,\sigma_2)$.

\begin{lemma} \label{lem:stationary}
Let $D\geq1$ and $0<\alpha<1/D^2$.
Let $G$ be an $n$-vertex connected graph which satisfies \ref{item:few suppressed} and \ref{item:few loaded}.
If $n$ is sufficiently large, we have that
\[d_{\mathrm{TV}}(\pi_G,\pi_{G^*})\le \exp(-(\log n)^{1/12}).\]
\end{lemma}

\begin{proof}
Let $U=U(\alpha,D)$, and let $U^*=f(U)\subseteq V(G^*)$.
We also fix $\tilde{G}\coloneqq
G\cup G^*$, and recall the distributions  defined in \eqref{eq:now}.
Observe that
\begin{equation}\label{eq:dTV1}
    \sum_{v\in V(\tilde G)}|{\deg_G(v)}-\deg_{G^*}(v)|=\sum_{v\in U}\deg_G(v)+\sum_{v\in U^*}\deg_{G^*}(v)\leq2\deg_G(U).
\end{equation}
The equality uses the fact that, for all $v\in V(G)\setminus U$, we have $\deg_{G}(v)=\deg_{G^*}(v)$.
In the inequality, we simply note that $\sum_{v\in U}\deg_G(v)=\deg_G(U)$ by definition and that, by the definition of $G^*$, this gives an upper bound for the second sum. 
Moreover, from \eqref{equa:dTVdef} and the triangle inequality we have
\begin{align}\label{eq:dTV2}
    d_{\mathrm{TV}}(\pi_G,\pi_{G^*})&=\frac12\sum_{v\in V(\tilde G)}\left|\frac{\deg_G(v)}{2e(G)}-\frac{\deg_{G^*}(v)}{2e(G^*)}\right|\nonumber\\
    &\leq\frac12\sum_{v\in V(\tilde G)}\left(\left|\frac{\deg_G(v)-\deg_{G^*}(v)}{2e(G)}\right|+\left|\frac{\deg_{G^*}(v)}{2e(G)}-\frac{\deg_{G^*}(v)}{2e(G^*)}\right|\right).
\end{align}
The second term in the sum can be evaluated as
\begin{align*}
    \sum_{v\in V(\tilde G)}\left|\frac{\deg_{G^*}(v)}{2e(G)}-\frac{\deg_{G^*}(v)}{2e(G^*)}\right|&=\sum_{v\in V(\tilde G)}\frac{e(G)-e(G^*)}{2e(G)e(G^*)}\deg_{G^*}(v)\\
    &=\frac{e(G)-e(G^*)}{e(G)}\\
    &=\frac{1}{2e(G)}\sum_{v\in V(\tilde G)}({\deg_G(v)}-\deg_{G^*}(v))\\
    &\leq\frac{1}{2e(G)}\sum_{v\in V(\tilde G)}|{\deg_G(v)}-\deg_{G^*}(v)|.
\end{align*}
Introducing this in \eqref{eq:dTV2} together with \eqref{eq:dTV1} and using Lemma~\ref{lem:smallU} and that $e(G)\geq n/2$, we conclude that 
\[d_{\mathrm{TV}}(\pi_G,\pi_{G^*})\leq\sum_{v\in V(\tilde G)}\left|\frac{\deg_G(v)-\deg_{G^*}(v)}{2e(G)}\right|\leq\frac{\deg_G(U)}{e(G)}\leq\exp(-(\log n)^{1/12}).\qedhere\]
\end{proof}

\subsection{Mixing time after contractions} \label{sec:contract}

The following result shows the mixing properties of contracted spreader graphs.
\begin{proposition} \label{prop:mixing G*}
For all $D\geq 4$, $0<\alpha<1/D^2$ and $\eps>0$, there exists a $C>0$ such that the following holds for all $n$ sufficiently large.
Suppose $G$ is an $n$-vertex connected $(\alpha,D)$-spreader graph.
Then,
\[t_\mix(G^*(\alpha,D),\eps)\leq C\log n.\]
\end{proposition}

In order to prove Proposition~\ref{prop:mixing G*}, we will rely on the following lemma.

\begin{lemma} \label{lem:G*conductance}
For all\/ $D\geq 4$ and\/ $0<\alpha<\beta=1/D^2$, the following holds for all\/ $n$ sufficiently large.
Suppose\/ $G$ is an\/ $n$-vertex connected\/ $(\alpha,D)$-spreader graph.
Then, taking\/ $G^*=G^*(\alpha,D)$, for all\/ $G^*$-connected\/ $S^*\subseteq V(G^*)$ such that\/ ${(\log n)^{1/2}}/{n}\leq \pi_{G^*}(S^*)\leq{1}/{2}$ we have that\/ $\Phi_{G^*}(S^*)\geq{\alpha^2}/{8}$.
\end{lemma}

\begin{proof}
Let $U\coloneqq U(\alpha,D)$ and $\gamma\coloneqq\alpha^2/4$.
Recall our definition of $f\colon V(G)\to V(G^*)$ from \eqref{eq:f}.
The proof will make use of the following claim.

\begin{claim}\label{claim:conductance2}
Any $G^*$-connected set $S^*\subseteq V(G^*)$ with $|S^*|\geq2$ and such that $(\log n)^{1/5}\leq|f^{-1}(S^*)|\leq(1-\beta)n$ is $\gamma$-good in $G^*$.
\end{claim}

\begin{claimproof}
Take any such $S^*\subseteq V(G^*)$ and let $S\coloneqq f^{-1}(S^*)=\{v\in V(G):f(v)\in S^*\}$.
Note that $S$ must be $G$-connected.
We claim that the bounds on $|S|$ imply that it must be $\gamma$-good in $G$.
Indeed, if $S$ was $\gamma$-bad in $G$, it would be a \mbox{$G$-connected} subset of $U$ (recall \eqref{eq:cS} and \eqref{eq: U def}) and so would be mapped by $f$ to a single vertex.
As $f$ is surjective and $|S^*|\geq2$, this is clearly not possible. 
It also follows from the definition of $G^*$ that $e_G(S)\ge e_{G^*}(S^*)$ and $|\del_G(S)|=|\del_{G^*}(S^*)|$. 
Therefore,
\[
\frac{|\partial_{G^*}(S^*)|}{\deg_{G^*}(S^*)}=
\frac{|\partial_{G^*}(S^*)|}{2e_{G^*}(S^*)+|\del_{G^*}(S^*)|}=\frac{|\partial_{G}(S)|}{2e_{G^*}(S^*)+|\del_{G}(S)|}\geq \frac{|\partial_{G}(S)|}{2e_G(S)+|\del_{G}(S)|}\geq \gamma, 
\]
and so $S^*$ is $\gamma$-good in $G^*$. 
\end{claimproof}

We will also need the fact that none of the sets from the statement are too large.

\begin{claim}\label{claim:conductance3}
Let $S^*\subseteq V(G^*)$ be a $G^*$-connected set such that $\pi_{G^*}(S^*)\leq 1/2$.
Then, $|S^*|\leq(1-2\beta)n$.
\end{claim}

\begin{claimproof}
Assume for a contradiction that there is such a set with $|S^*|>(1-2\beta)n$.
Letting $\bar S^*\coloneqq V(G^*)\setminus S^*$, we have that $\pi_{G^*}(S^*)+\pi_{G^*}(\bar S^*)=1$ and $|\del_{G^*}(S^*)|=|\del_{G^*}(\bar S^*)|$, implying that
\[e_{G^*}(\bar S^*)\geq e_{G^*}(S^*)\geq |S^*|-1\geq (1-3\beta)n>3n/4,\]
where we used here that $S^*$ is $G^*$-connected and the fact that $\beta=1/D^2\le 1/16$.
Let $T\subseteq V(G)$ be some set of size $3\beta n$ such that $f^{-1}(\bar S^*)\subseteq T$, noting that this is possible since, by Lemma~\ref{lem:smallU}, we have $|f^{-1}(\bar S^*)|\le |\bar S^*|+|U|\le 3 \beta n$.
Now, by \ref{item: no dense set}, we have that $e_{G^*}(\bar S^*)\le e_G(T)\le D|T|\le 3D\beta n\le 3n/4$, a contradiction, where we again appealed to the facts that $\beta=1/D^2$ and $D\geq 4$.
\end{claimproof}

Now suppose that there exists some $G^*$-connected set $S^*\subseteq V(G^*)$ with ${(\log n)^{1/2}}/{n}\leq \pi_{G^*}(S^*)\leq1/2$ and $\Phi_{G^*}(S^*)<\alpha^2/8$.
By the bound on the conductance, it follows from \eqref{eq:conductance simplified} that $S^*$ is $\gamma$-bad in $G^*$ and so, by Lemma~\ref{lem:simplebad}\ref{item:bad1}, we have $|\del_{G^*}(S^*)|\leq 2e_{G^*}(S^*)$.
This implies that
\begin{align}\label{eq:conductance1}
    4e_{G^*}(S^*)\geq \deg_{G^*}(S^*) &= \pi_{G^*}(S^*)\cdot 2e(G^*) \nonumber \\
    &\ge \frac{(\log n)^{1/2}}{n}2e(G^*) \nonumber\\ 
    &\ge \frac{(\log n)^{1/2}(2e(G)-\deg_G(U))}{n}\nonumber\\
    &\ge \frac{ (\log n)^{1/2}}{2},
\end{align}
where in the last inequality we used the fact that $\deg_G(U)=o(n)$ from Lemma~\ref{lem:smallU} and the fact that $e(G)\geq n/2$ as $G$ is connected.

Observe that, since $\gamma<1$ and $G^*$ is connected, no set $S^*\subseteq V(G^*)$ with $|S^*|=1$ can be $\gamma$-bad, so we must have $|S^*|\geq2$.
Then, Claim~\ref{claim:conductance2} implies that $S\coloneqq f^{-1}(S^*)$ has size $|S|<(\log n)^{1/5}$ or $|S|>(1-\beta)n$.
If $|S|>(1-\beta)n$, then $|S^*|\geq|S|-|U|>(1-2\beta)n$ (again by Lemma~\ref{lem:smallU}), and we know this cannot happen by Claim~\ref{claim:conductance3}, so we must have $|S|<(\log n)^{1/5}$.
As, trivially, any vertex set $S'\subseteq V(G)$ has $e_{G}(S')\le |S'|^2$, we have that $e_{G^*}(S^*)\leq e_G(S)\leq(\log n)^{2/5}$.
This contradicts \eqref{eq:conductance1}.
\end{proof}

With this, we can prove Proposition~\ref{prop:mixing G*}.

\begin{proof}[Proof of Proposition~\ref{prop:mixing G*}]
Let $G^*\coloneqq G^*(\alpha,D)$.
Observe that \ref{item: no dense set} implies $e(G^*)\leq e(G)\leq Dn$. 
By Lemma~\ref{lem:G*conductance}, for any $G^*$-connected set $S^*\subseteq V(G^*)$ with $(\log n)^{1/2}/n\leq\pi_{G^*}(S^*)\leq1/2$ we have $\Phi_{G^*}(S^*)\geq\alpha^2/8$.
Consider now any $G^*$-connected set $S^*\subseteq V(G^*)$ with $\pi_{G^*}(S^*)\leq(\log n)^{1/2}/n$ (in particular, $S^*\neq V(G^*)$).
The fact that $G^*$ is connected together with \eqref{equa:conductdef1} and \eqref{equa:conductdef2} ensures that
\[\Phi_{G^*}(S^*)\geq\frac{|\partial_{G^*}(S^*)|}{4e(G^*)\cdot\pi_{G^*}(S^*)}\geq\frac{1}{4Dn\pi_{G^*}(S^*)}.\]
Let $J$ denote the set of indices $j\geq1$ such that $2^{-j}\leq(\log n)^{1/2}/n$, and note that $\min(J)\leq\log n$. 
It then follows that
\begin{align*}
    \sum_{j=1}^{\lceil\log_2\pi_{\min}(G^*)^{-1}\rceil}\Phi_{G^*}^{-2}(2^{-j})&\leq\log n\cdot\frac{64}{\alpha^4}+\sum_{j\in J}2^{-2j}(4Dn)^2\\
    &\leq\log n\cdot\frac{64}{\alpha^4}+2\max_{j\in J}\{2^{-2j}\}\cdot16D^2n^2\leq\left(\frac{64}{\alpha^4}+32D^2\right)\log n,
\end{align*}
where in the last inequality we use the definition of $J$.
By Theorem~\ref{thm:FR} we have
\[
t_{\textrm{mix}}(G^*)\leq C_0\left(\frac{64}{\alpha^4}+32D^2\right)\log n,
\]
where $C_0$ is some absolute constant.
Since the total variation distance decreases exponentially fast after the mixing time (see, e.g.,~\cite[section~4.5]{LPW09}), we get
\[
t_{\textrm{mix}}(G^*,\eps)\leq C_0\left(\frac{64}{\alpha^4}+32D^2\right)\lceil\log_2(1/\eps)\rceil\log n,
\]
and the proposition holds by taking $C$ appropriately.
\end{proof}

\subsection{Hitting time of bad vertices} \label{sec:hitting}

Let $D\geq4$ and $0<\alpha<1/D^2$ and consider a connected $(\alpha,D)$-spreader graph $G$.
We now wish to study the hitting time to the set of bad vertices $U(\alpha,D)$ in $G$ and show that a.a.s.\ it is not too small.

\begin{lemma} \label{lem:hitting time}
Let $D\geq 4$ and $0<\alpha<1/D^2$.
Let $G$ be an $n$-vertex connected $(\alpha,D)$-spreader graph, and let $U=U(\alpha,D)$. 
%Let $(X_t)_{t\geq0}$ be the lazy random walk on $G$ starting at $v\in V(G)\setminus U$.
Then,
\[\frac{1}{n}\sum_{v\in V(G)\setminus U}\PP[\tau_G(\mu_0^v,U)\leq(\log n)^2]=o(1).\]
\end{lemma}

\begin{proof}
In this proof we will use a new auxiliary multigraph $\hat G$.
Let us introduce it here.
Consider the multigraph $G^*=G^*(\alpha,D)$, and let $U^*\coloneqq f(U)\subseteq V(G^*)$.
Recall that the definition of $U^*$ implies that it is an independent set in $G^*$.
Now, $\hat G$ is obtained from $G^*$ by contracting $U^*$ to a single new vertex $u^*$.
The fact that $U^*$ is an independent set in $G^*$ guarantees that $e(G^*)=e(\hat G)$, and since vertices outside $U^*$ do not see their degree changed by this operation, it follows that $\pi_{\hat G}(v)=\pi_{G^*}(v)$ for all $v\in V(G)\setminus U$ and $\pi_{\hat G}(u^*)=\pi_{G^*}(U^*)$.

Let $(X_t)_{t\geq0}$ and $(\hat X_t)_{t\geq0}$ denote lazy random walks on $G$ and $\hat G$ starting on some vertex $v\in V(G)\setminus U$.
Let us denote $\tau^v_U\coloneqq\tau_G(\mu_0^v,U)$ and $\hat\tau^v_{u^*}\coloneqq\tau_{\hat G}(\mu_0^v,u^*)$.
Define the natural coupling $(X_t,\hat X_t)_{t\geq 0}$ as follows: for any $t\geq 1$, while $X_t\notin U$, let $\hat X_t=X_t$; if there is a $t\geq 1$  such that $X_t\in U$, then for the smallest such $t$ we let $\hat X_t=u^*$; otherwise (that is, for all $t>\tau^v_U$), we let $X_t$ and $\hat X_t$ evolve independently.
Observe that this is indeed a valid coupling since for all $v\in V(G)\setminus U$ we have $e_G(v,U)=e_{\hat G}(v,u^*)$.
With this natural coupling, conditional on $\tau^v_U>t$, we have  $X_t=\hat X_t$ and so $\tau^v_U=\hat \tau^v_{u^*}$; in particular, for any $T_0\geq 0$ we have
\begin{equation}\label{eq:equivalence}
\PP[\tau^v_U>T_0]=\PP[\hat\tau^v_{u^*}> T_0].
\end{equation}

We will study the hitting times $\hat\tau^v_{u^*}$ using Theorem~\ref{thm:MQS} on $\hat{G}$ with $T=T(n)=(\log{n})^6$ and $c=3$.
We start with the following claim (which we make no efforts to optimise). 

\begin{claim}\label{claim:mixing hat G}
We have that $t_\mix (\hat G)\leq (\log n)^4$.
\end{claim}

\begin{claimproof}
First we prove that, for any $\hat G$-connected set $\hat S\subseteq V(\hat G)$ such that $\pi_{\hat G}(\hat S)\leq 1/2$, we have that $\Phi_{\hat G}(\hat S)\geq 1/(2\log n)$.
Indeed, fix such an $\hat S$ and define $S^*\subseteq V(G^*)$ as 
\begin{equation*} \label{eq:S*}
S^*= \begin{cases}
    \hat S & \text{if } u^*\notin \hat S, \\
    (\hat S\setminus\{u^*\})\cup U^* & \text{if } u^*\in \hat S.
\end{cases}
\end{equation*}
Further, let $S^*=S_1\cup \ldots \cup S_r$ for some $r\in \NN$ be a decomposition of $S^*$ into $G^*$-connected components (note that $r=1$ if $u^*\notin \hat S$). 
Now, as $\deg_{\hat G}(\hat S)=\deg_{G^*}(S^*)$ and $e(\hat G)=e(G^*)$, we have that $\pi_{G^*}(S^*)\leq 1/2$ and hence $\pi(S_i)\leq 1/2$ for all $i\in [r]$.
Moreover, as $U^*$ is an independent set in $G^*$, we have that $|\partial_{\hat G}(\hat S)|=\sum_{i=1}^r|\partial_{G^*}(S_i)|$ and $\deg_{\hat G}(\hat S)=\sum_{i=1}^r \deg_{G^*}(S_i)$. 

Returning to analyse the conductance of $\hat S$, from \eqref{eq:conductance simplified}, we have that 
\begin{equation} \label{eq:cond S hat 1}
    \Phi_{\hat G}(\hat S)\geq \frac{|\partial_{\hat G}(\hat S)|}{2 \deg_{\hat G}(\hat S)}=\frac{\sum_{i=1}^r|\partial_{G^*}(S_i)|}{2\sum_{i=1}^r\deg_{G^*}(S_i)}\geq \frac{1}{2}\min_{i\in [r]}\frac{|\partial_{G^*}(S_i)|}{\deg_{G^*}(S_i)}=\frac{|\partial_{G^*}(S_{i_0})|}{2\deg_{G^*}(S_{i_0})},
\end{equation}
letting $i_0\in [r]$ be a minimising index.
If $\deg_{G^*}(S_{i_0})\leq \log n$, then we are done due to the fact that $|\partial_{G^*}(S_{i_0})|\geq 1$ as $G^*$ is connected. 
If $\deg_{G^*}(S_{i_0})> \log n$, then $\pi_{G^*}(S_{i_0})= \deg_{G^*}(S_{i_0})/(2e(G^*))\geq (\log n)^{1/2}/n$, using that $e(G^*)\le e(G)\le Dn$ from \ref{item: no dense set}.
Therefore, using \eqref{eq:conductance simplified} and \eqref{eq:cond S hat 1}, we have that
\[\Phi_{\hat G}(\hat S)\geq \frac{|\partial_{G^*}(S_{i_0})|}{2\deg_{G^*}(S_{i_0})} =\frac{\Phi_{G^*}(S_{i_0})\deg_{G^*}(V(G^*)\setminus S_{i_0})}{2e(G^*)}\geq \frac{\Phi_{G^*}(S_{i_0})}{2}\geq \frac{\alpha^2}{16},\]
where we used Lemma~\ref{lem:G*conductance} in last inequality and the fact that $\pi_{G^*}(S_{i_0})=1-\pi_{G^*}(V(G^*)\setminus S_{i_0})\leq 1/2$ in the penultimate inequality.

So we have established that $\Phi_{\hat G}(\hat S)\geq 1/(2 \log n)$ for all $\hat G$-connected sets $\hat S$ with $\pi_{\hat G}(\hat S)\leq 1/2$.
Now notice that $\pi_{\min}(\hat G)\geq 1/(2e(\hat G))\geq 1/(2Dn)$ due to \ref{item: no dense set}, and hence, when applying Theorem~\ref{thm:FR}, there are logarithmically many terms in the sum.
This establishes the desired upper bound on $t_{mix}(\hat G)$.
\end{claimproof}

Using \eqref{equa:dTVdef} and Claim~\ref{claim:mixing hat G} and as the total variation distance decreases exponentially fast after the mixing time~(see, e.g.,~\cite[section~4.5]{LPW09}), we have
\[
\max_{x,y\in V(\hat G)}|\mu_{0}^xP_{\hat G}^T(y)-\pi_{\hat G}(y)|\leq 2 \max_{x\in V(\hat G)} d_{\mathrm{TV}}(\mu_0^x P^T_{\hat G},\pi_{\hat G})= o(n^{-3})
\] 
and \ref{item:HP1} is satisfied.

Let us now prove that $\pi_{\hat G}(u^*)$ is small.
By Remark~\ref{rem:pi_max_not_needed}, to prove our statement it suffices to have \ref{item:HP2'} for $u=u^*$.
By Lemma~\ref{lem:smallU} \eqref{eq:smallU3}, $\pi_G(U)=o((\log n)^{-6})$.
Moreover, as mentioned earlier, $\pi_{\hat G}(u^*)=\pi_{G^*}(U^*)$.
By Lemma~\ref{lem:stationary}, we have that 
\begin{align}\label{eq:bound_pi_u*}
    \pi_{\hat G}(u^*)= \pi_{G}(U)+O(d_{\mathrm{TV}}(\pi_G,\pi_{G^*}))= o((\log n)^{-6}).
\end{align}
It follows that $T\cdot\pi_{\hat G}(u^*)= o(1)$ and \ref{item:HP2'} holds.

Finally, since $G$ is a connected $(\alpha,D)$-spreader graph with $D$ fixed and $\alpha<1$, $\hat G$ is a connected multigraph with $e(\hat{G})\leq e(G)\leq Dn$, by \ref{item: no dense set}.
It follows by \eqref{equa:statdistdef1} that $\pi_{\min}(\hat G)\geq 1/2e(\hat G)=\omega(n^{-2})$ and \ref{item:HP3} is satisfied.

Now, recalling the relevant definitions from Theorem~\ref{thm:MQS} and letting $T_0\coloneqq\lceil(\log{n})^2\rceil$, as $n$ goes to infinity we have
\[\lambda_{u^*}^{T_0}= \left(1-(1+o(1))\frac{\pi_{\hat{G}}(u^*)}{R_{T}(u^*)}\right)^{T_0}\geq 1-(1+o(1))\frac{T_0\pi_{\hat{G}}(u^*)}{R_{T}(u^*)}= 1-o(1),\]
where we used $(1-x)^{T_0}\geq 1-T_0x$, \eqref{eq:bound_pi_u*} and $R_{T}(u^*)\geq 1$.
Therefore, appealing to Remark~\ref{rem:replace pi}, we have that 
\begin{align*}
\frac{1}{n}\sum_{v\in V(G)\setminus U}\PP[\hat\tau^v_{u^*}\leq  T_0 ] &\leq  \frac{1}{|V(\hat{G})|}\sum_{v\in V(G)\setminus U}\PP[\hat\tau^v_{u^*}\leq  T_0 ]\\
&\leq \frac{1}{|V(\hat{G})|}\sum_{v\in V(\hat{G})}\PP[\hat\tau^v_{u^*}\leq  T_0 ] =o(1).
\end{align*}
%where we appealed to Lemma~\ref{lem:smallU}~\eqref{eq:smallU1} in the first inequality. 
Together with \eqref{eq:equivalence}, this completes the proof of the lemma.
\end{proof}

\subsection{Proof of the main theorem}\label{sec:mainproof}

We are finally ready to prove Theorem~\ref{thm:main}.

\begin{proof}[Proof of Theorem~\ref{thm:main}]
Let $U=U(\alpha,D)$ and $G^*=G^*(\alpha,D)$.
Let $(X_t)_{t\geq0}$ and $(X^*_t)_{t\geq0}$ denote lazy random walks on $G$ and $G^*$, respectively.
For $x\in V(G)$, let $\tau^x_U\coloneqq\tau_G(\mu_{0}^x,U)$.
Similarly as in the proof of Lemma~\ref{lem:hitting time}, we consider a natural coupling $(X_t,X^*_t)_{t\geq0}$ of the random walks so that for $t<\tau_U$ we let $X^*_t=X_t$ and otherwise we let the walks evolve independently.

First, observe that, by Lemma~\ref{lem:stationary} and the triangle inequality (and adopting the abuse of notation introduced in \cref{sec:stationary}), for any $x\in V(G)$ we have
\begin{align}\label{eq:traing_ineq}
    d_{\textrm{TV}}(\mu_0^x P^t_G,\pi_G)\leq d_{\textrm{TV}}(\mu_0^x P^t_G,\pi_{G^*})+ d_{\textrm{TV}}(\pi_G,\pi_{G^*})= d_{\textrm{TV}}(\mu_0^x P^t_G,\pi_{G^*})+ o(1).
\end{align}

For every $x\in V(G)\setminus U$ and $y\in V(G)$, we can write
\begin{align*}
    \mu_0^x P^t_G(y) 
     &= \PP[X_t=y\mid X_0=x]\nonumber\\
     &= \PP[X_t=y,\tau_U>t\mid X_0=x] +\PP[X_t=y,\tau_U\leq t\mid X_0=x]\nonumber\\
     &= \PP[X^*_t=y,\tau_U>t\mid X_0=x] +\PP[X_t=y,\tau_U\leq t\mid X_0=x]\nonumber\\
     &\leq \mu_{0}^x P^t_{G^*}(y) +\PP[X_t=y,\tau_U\leq t\mid X_0=x],
\end{align*}
where we let $\mu_{0}^x P^t_{G^*}(y)=0$ if $y\in U$. 
Let $A\coloneqq\{y\in V(G): \mu_{0}^x P^t_{G}(y)\geq \pi_{G^*}(y)\}$.
It follows that 
\begin{align}\label{eq:from_G_to_G^*2}
    d_{\textrm{TV}}(\mu_{0}^x P^t_G,\pi_{G^*})
     &= \sum_{y\in A} (\mu_{0}^xP^t_{G}(y)- \pi_{G^*}(y))\nonumber\\
     &\leq \sum_{y\in A} (\mu_{0}^xP^t_{G^*}(y)- \pi_{G^*}(y)) + \sum_{y\in A} \PP[X_t=y,\tau_U\leq t\mid X_0=x]\nonumber\\
     &\leq d_{\textrm{TV}}(\mu_{0}^x P^t_{G^*},\pi_{G^*}) + \PP[\tau^x_U\leq t].
\end{align}

Let $C>0$ be the constant given by Proposition~\ref{prop:mixing G*} with $\eps/2$ playing the role of $\eps$, and let $T_0\coloneqq \lceil C\log {n}\rceil$. 
By Proposition~\ref{prop:mixing G*}, we have that $d_{\textrm{TV}}(\mu_{0}^x P^{T_0}_{G^*},\pi_{G^*})\leq\eps/2$.
Combining these bounds with \eqref{eq:traing_ineq} and \eqref{eq:from_G_to_G^*2}, we obtain for all $x\in V(G)\setminus U$ that
\begin{align*}
    d_{\textrm{TV}}(\mu_{0}^x P^{T_0}_G,\pi_G)\leq \eps/2+\PP[\tau^x_U\leq T_0]+o(1).
\end{align*}
By Lemma~\ref{lem:smallU}~\eqref{eq:smallU1} and Lemma~\ref{lem:hitting time}, we conclude that
\begin{align*}
    \frac{1}{n}\sum_{x\in V(G)} d_{\textrm{TV}}(\mu_{0}^x P^{T_0}_G,\pi_G)\leq \frac{|U|}{n}+\frac{n-|U|}{n}\left(\frac{\eps}{2}+o(1)\right)+\frac{1}{n}\sum_{x\in V(G)\setminus U}\PP[\tau^x_U\leq T_0] \leq \eps
\end{align*}
and $\bar t_{\mix}(G,\eps)\leq T_0$, concluding the proof.
\end{proof}

\section{Smoothed analysis on connected graphs} \label{sec:smoothed}

We next want to show applications of Theorem~\ref{thm:main}.
We use this section to prove Theorem~\ref{thm:smoothed}.

\begin{proof}[Proof of Theorem~\ref{thm:smoothed}]
Let $D\coloneqq2(\Delta +1+\delta)$  and $0<\alpha<\delta/D^2$ be a sufficiently small constant. 
If $G'$ is a connected $(\alpha,D)$-spreader graph, the statement follows from Theorem~\ref{thm:main}.
Since $G'$ is connected by assumption, it suffices to show that a.a.s.\ $G'$ is an $(\alpha,D)$-spreader graph.

We need to verify that \ref{item:few suppressed}, \ref{item:few loaded} and \ref{item: no dense set} hold a.a.s.\ in $G'=G\cup R$.
The fact that \ref{item:few loaded} and \ref{item: no dense set} hold a.a.s.\ follows directly from the fact that $G$ is $\Delta$-degenerate and that $e(R[S])<\max\{2,2\delta\}|S|$ for all $S\subseteq V(R)$ (see, for example, \cite[Lemma 8]{KRS15}). 
The fact that property~\ref{item:few suppressed} holds a.a.s.\ in $G'$ follows from the proof of \cite[Theorem~3]{KRS15}.
Indeed, for each $(\log n)^{1/5}\leq k\leq(1-1/D^2)n$, let $X_k$ denote the number of $G'$-connected $\alpha$-thin sets $S\subseteq V(G)$ with $|S|=k$.
Then (see~\cite[eq.~(4)]{KRS15} and the claim immediately after), one can check that the expected number of such sets satisfies 
\[\mathbb{E}[X_k]\leq\sum_{m=1}^{\alpha k}\sum_{b=m}^{\alpha k}n^m\exp\left(C\alpha\log(1/\alpha)k\right)\left(\frac{\delta}{n}\right)^{m-1}\mathbb{P}\left[\mathrm{Bin}\left(k(n-k),\frac{\delta}{n}\right)<\alpha k\right],\]
where $C$ is some absolute constant.
Following again the proof in~\cite{KRS15} and choosing $\alpha$ sufficiently small and $n$ sufficiently large, one concludes that
\[\mathbb{E}[X_k]\leq k^2n\exp\left(\left(C\alpha\log(1/\alpha)-\frac{\delta}{8D^2}\right)k\right)\leq n\exp\left(-\frac{\delta}{20D^2}k\right).\]
Property \ref{item:few suppressed} then follows by Markov's inequality and a union bound over all $(\log n)^{1/5}\leq k\leq(1-1/D^2)n$. 
\end{proof}

\section{Random subgraphs of expanders}\label{sec:expanders}

In order to prove Theorem~\ref{thm:expanders}, we will rely on several known properties of the giant component of a random subgraph of an $(n,d,\lambda)$-graph.
Recall from Section \ref{sec:notation} that when we refer to asymptotic statements holding in \emph{an} $(n,d,\lambda)$-graph, implicitly what is meant is that the  statement holds for any sequence $(G_n)_{n\geq 1}$ of $(n,d,\lambda)$-graphs that satisfy the stated condition.

\begin{lemma}\label{lem:expanderaux}
Let $\delta>0$ be a sufficiently small constant and let $G$ be an $(n,d,\lambda)$-graph with $\lambda\leq\delta^4d$.
Let $p=(1+\delta)/d$ and let $L_1$ be a largest component in $G_p$.
Then, a.a.s.\ the following properties hold:
\begin{enumerate}[label = (\alph*)]
    \item \label{list:a} $L_1$ has $(1+o(1))(2\delta+g(\delta)) n$ vertices, where $g(\delta)=o(\delta)$  as $\delta$ tends to $0$. 
    \item \label{list:b} There exists some absolute constant $c>0$ such that, for any $G_p$-connected $S\subseteq V(L_1)$ with $16 \log n/\delta^2\leq |S|\leq\delta^2n/50$, we have that 
    \begin{equation*}\label{cond:b}
        |\del_{G_p}(S)|\ge \frac{c\delta^2|S|}{\log(1/\delta)}.
    \end{equation*}
    \item \label{list:c} There exists some absolute constant $c'>0$ such that, for \emph{any} $S\subseteq V(L_1)$ with ${\delta^2n}/{50}\le |S|\le {12\delta n}/{11}$, we have that 
    \begin{equation*}\label{cond:c}
        |\del_{G_p}(S)|\ge \frac{c'\delta^2|S|}{\log(1/\delta)}.
    \end{equation*}
\end{enumerate}
\end{lemma}

\begin{proof}
Statement \ref{list:a} follows from \cite[Theorem~1]{FKM04}; see also the discussion following \cite[Theorem~1.1]{DK22}.
Statement \ref{list:b} is a consequence of \cite[Theorem~1~(1)]{DK22},  while \ref{list:c} is given in \cite[Theorem~2]{DK22}.
\end{proof}

With this, we can prove Theorem~\ref{thm:expanders}.

\begin{proof}[Proof of Theorem~\ref{thm:expanders}]
Let $L_1\coloneqq L_1(G_p)$.
Fix $D\coloneqq12$, $c_0\coloneqq\min\{c,c'\}$ (where $c$ and $c'$ are the absolute constants from Lemma~\ref{lem:expanderaux}\ref{list:b} and \ref{list:c}) and $\alpha\coloneqq c_0\delta^2/(D^2\log(1/\delta))$. 
By Theorem~\ref{thm:main}, it suffices to show that a.a.s.\ $L_1$ is an $(\alpha,D)$-spreader graph.
That is, we need to show that (for sufficiently small $\delta>0$) a.a.s.\ $L_1$ satisfies properties \ref{item:few suppressed}--\ref{item: no dense set}.

We are first going to show that \ref{item:few suppressed} and \ref{item:few loaded} hold a.a.s.\ in $G_p$, rather than $L_1$, for sets of size at least $(\log n)^{1/6}$.
Similarly as happened in the proof of Theorem~\ref{thm:smoothed}, in this case properties \ref{item:few suppressed} (for $|S|\leq \delta^2 n/50$) and \ref{item:few loaded} can be obtained by following the proofs of \cite[Theorem~1~(1)]{DK22} and \cite[Lemma~2.4]{DK22}, respectively.
Let us give here a brief sketch.
Let us first consider \ref{item:few suppressed} (for $|S|\leq\delta^2n/50$).
For each $(\log n)^{1/6}\leq k\leq\delta^2n/50$, let $X_k$ denote the number of $G_p$-connected sets $S\subseteq V(G)$ with $|S|=k$ which are $\alpha$-thin.
Then, following \cite[Theorem~1~(1)]{DK22}, we have $\mathbb{E}[X_k]\leq3n\exp(-\delta^2k/8)$.
By Markov's inequality and a union bound over all $(\log n)^{1/6}\leq k\leq\delta^2n/50$, we conclude that a.a.s.
\begin{enumerate}[label=(d)]
    \item\label{item:final1} for all $(\log n)^{1/6}\leq k \leq\delta^2 n/50$, the number of $G_p$-connected $\alpha$-thin sets $S\subseteq V(G)$ with $|S|=k$ is less than $n\nume^{-\sqrt{k}}$.
\end{enumerate}
Similarly, for each $k\geq(\log n)^{1/6}$, one can check from the proof of \cite[Lemma~2.4]{DK22} that, if we let $Y_k$ denote the number of $G_p$-connected sets $S\subseteq V(G)$ with $|S|=k$ such that $e_{G_p}(S)\geq10k$, then $\mathbb{E}[Y_k]\leq n\exp(-2k)$.
Again by Markov's inequality and a union bound over all $k\geq(\log n)^{1/6}$, we conclude that a.a.s.
\begin{enumerate}[label=(e)]
    \item\label{item:final2} for all $(\log n)^{1/6}\le k \leq n$, the number of $G_p$-connected $\alpha^{-1}$-loaded sets  $S\subseteq V(G)$ with $|S|=k$ is less than $n\nume^{-\sqrt{k}}$.
\end{enumerate}

We next work towards property \ref{item: no dense set}.
We are going to show that no set $S\subseteq V(G)$ with $|S|\ge \delta \alpha n$ is $D$-loaded in $G_p$.
Fix some $\delta \alpha n\le k \le n$ and let $\eta\coloneqq k/n$.
If $\delta>0$ is sufficiently small, we have that $\delta^4<\alpha\delta\leq\eta$.
Now, by the expander mixing lemma (see, for example, \cite[Lemma 2.1]{DK22}), for any set $S\subseteq V(G)$ with $|S|=k=\eta n$ we have that
\[e_G(S)\le \frac{dk^2}{n}+\lambda k\le (\eta +\delta^4) dk\le 2\eta dk,\]
using that $\lambda\leq\delta^4d$.
Hence, the probability that $S$ contains $Dk$ edges in $G_p$ is at most 
\[\binom{2\eta dk}{Dk}p^{Dk}\leq\left(\frac{2\eta d \nume}{D}\right)^{Dk}\left(\frac{1+\delta}{d}\right)^{Dk}\leq\left(\frac{6\eta}{D}\right)^{Dk}.\]
Taking a union bound over all possible sets of size $k$, we have that the probability of there existing a $D$-loaded set of size $k$ in $G_p$ is at most 
\[\binom{n}{k}\left(\frac{6\eta}{D}\right)^{Dk}\leq\left(\frac{n\nume}{k}\right)^k\left(\frac{6\eta}{D}\right)^{Dk}=\left(\frac{\nume}{\eta}\left(\frac{6\eta}{D}\right)^{D}\right)^k=\left(\nume\eta^{D-1}\left(\frac{1}{2}\right)^{D}\right)^k\leq\left(\frac{1}{2}\right)^{\alpha\delta n}.\]
By a union bound over all $\alpha \delta n\le k\le n$ we conclude that a.a.s.
\begin{enumerate}[label=(f)]
    \item\label{item:final3} no set $S\subseteq V(G)$ with $|S|\geq\delta\alpha n$ is $D$-loaded in $G_p$.
\end{enumerate}

Condition on the event that Lemma~\ref{lem:expanderaux}\ref{list:a} and \ref{list:c} as well as \ref{item:final1}, \ref{item:final2} and \ref{item:final3} hold in $G_p$, which a.a.s.\ occurs.
Let $n'\coloneqq|V(L_1)|=(1+o(1))(2\delta+g(\delta))n$ by \ref{list:a}.
It follows that $(\log n')^{1/5}\geq(\log n)^{1/6}$ and so, by \ref{item:final1}, \ref{item:few suppressed} holds in $L_1$ for sets of size at most $\delta^2n/50$; similarly, \ref{item:few loaded} holds by \ref{item:final2}, and \ref{item: no dense set} holds by \ref{item:final3}. 
Thus, it only remains to establish \ref{item:few suppressed} for sets $S\subseteq V(L_1)$ such that $\delta^2 n/50\leq|S|\leq(1-1/D^2)n'$.
For $\delta^2 n/50\leq|S|\leq12\delta n/11$, this is immediate from \ref{list:c}, and we can also use \ref{list:c} for larger sets $S$.
Indeed, suppose ${12\delta n}/{11}\leq|S|\leq(1-1/D^2)n'$ and let $\bar S\coloneqq V(L_1)\setminus S$.
It follows from \ref{list:a}, by taking a sufficiently small $\delta$,  that
\[|\bar{S}|=|L_1|-|S|\leq n'-\frac{12}{11}\delta n
%= n'-\frac{12}{11}\cdot\frac{\delta}{2\delta+g(\delta)}n'
<\frac{n'}{2}\le \frac{12}{11}\delta n,\]
where in the last two inequalities we use that  $\frac{\delta}{2\delta+g(\delta)}$ tends to $1/2$  as $\delta$ tends to $0$ due to the fact that $g(\delta)=o(\delta)$.
We also have that $|\bar S|\geq n'/D^2\geq|S|/D^2\geq \delta^2 n/50$, using also here that $\delta$ is sufficiently small.
Hence, we can apply \ref{list:c} to $\bar S$ and we obtain that 
\[|\del_{L_1}(S)|=|\del_{L_1}(\bar S)|\ge \frac{c'\delta^2|\bar S|}{\log(1/\delta)}\ge \frac{c'\delta^2|S|}{D^2\log(1/\delta)}\geq\alpha|S|.\qedhere\]
\end{proof}

\section{Open problems} \label{sec:conc}

Theorem~\ref{thm:main} is only effective on graphs where the mixing is slowed down by few small bottlenecks.
This is the case in the two applications presented.
Nevertheless, there are other cases where both small and large bottlenecks exist.
It would be interesting to study average-case mixing times in such scenarios and determine which improvement with respect to the worst-case can be attained.
One such example is the small-world model of Kleinberg~\cite{K00}, whose mixing time has been studied in~\cite{DGGJV20}.

In recent years, the theory of random walks in random directed graphs has attracted a considerable amount of attention.
As in the case of random regular graphs, under mild conditions on the bidegree sequence, the mixing time is logarithmic~\cite{BCS18,CCPQ22}.
From the point of view of smoothed analysis, a natural question is whether randomly perturbing a deterministic strongly connected digraph can yield logarithmic mixing time.
Conductance-based bounds such as Jerrum-Sinclair and Fountoulakis-Reed are not valid in the non-reversible setting, which requires new ideas.
Finally, we mention an analogous result to the mixing time in randomly perturbed connected graphs by Krivelevich, Reichman and Samotij~\cite{KRS15}, for graphs perturbed by a random perfect matching, has been obtained by Hermon, Sly and Sousi~\cite{HSS22}.
Considering such a model in the directed setting would also be interesting.

\medskip

\noindent\textbf{Acknowledgements.} The authors would like to thank Matteo Quattropani for fruitful discussions on the First Visit Time Lemma (FVTL).  They would also like to thank the anonymous referees for their insightful comments, in particular for spotting a misuse of the FVTL and for pointing out the non-contractivity of the average mixing time~(see Remark~\ref{rem:contractive}).

% Use with natbib, not biblatex:
\bibliographystyle{mystyle} 
\bibliography{RandWalkBib}

% Use with biblatex, not natbib:
% \printbibliography

\end{document}